\documentclass[11pt]{article}
\usepackage{amsfonts,latexsym,rawfonts,amsmath,amssymb,amsthm,graphicx}
\numberwithin{equation}{section}
\newtheorem{theorem}{Theorem}[section]
\newtheorem{lem}[theorem]{Lemma}
\newtheorem{thm}[theorem]{Theorem}

\newtheorem{cor}[theorem]{Corollary}

\def\endproof{$\hfill\Box$\\}
\def\R{\mathbb{R}}

\title{Heat flow of extrinsic biharmonic maps from a four dimensional manifold with boundary}
\author{Tao Huang, Lei Liu, Yong Luo, Changyou Wang}
\begin{document}
\maketitle
\begin{abstract}
Let $(M,g)$ be a four dimensional compact Riemannian manifold with boundary and $(N,h)$ be a compact Riemannian manifold without boundary. We show the existence of a unique, global weak solution of the heat flow of extrinsic biharmonic maps from $M$ to $N$ under the Dirichlet boundary condition, which is regular with the exception of at most finitely many time slices. We also discuss the behavior of solution near the singular times.
 As an immediate application,  we prove the existence of a smooth extrinsic biharmonic map from
 $M$ to $N$ under any Dirichlet boundary condition.
\end{abstract}

\section{Introduction}

Let $(M,g)$ be a Riemannian manifold with or without boundary and $(N,h)$ a Riemannian manifold without boundary and isometrically embedded in $\mathbb{R}^L$. For a nonnegative integer $l$ and $1\leq p<\infty$, the Sobolev space $W^{l,p}$ and H\"older space $C^{l+\alpha}(0<\alpha<1)$ are defined by:
\begin{eqnarray*}
&&W^{l,p}(M,N):=\Big\{u\in W^{l,p}(M,\R^L)\ |\ u(x)\in N ~{\rm{for ~a.e.~}} x\in M\Big\},
\\
&&C^{l+\alpha}(M,N):=\Big\{u\in C^{l,\alpha}(M,\R^L)\ |\ u(x)\in N~\forall x\in M\Big\}.
\end{eqnarray*}
On $W^{2,2}(M,N)$, there are two natural second order energy functionals defined by
$$F_2(u)=\int_M|\Delta u|^2dv_g,~~E_2(u)=\int_M|\tau(u)|^2dv_g,$$
where $\Delta$ is the Laplace-Beltrami operator of $(M,g)$, $$\tau(u)=\Delta u+A(u)(\nabla u,\nabla u)$$ is the tension field of $u$, and $A(\cdot)(\cdot,\cdot)$ is the second fundamental form of $(N,h)$ in $\R^L$.

A map is called an extrinsic (or intrinsic, resp.) biharmonic map if $u$ is a critical point of $F_2$ (or~ $E_2$, resp.). The Euler-Lagrange equation for $F_2$ is (cf. \cite{CWY,Lamm4,Lamm1,W2,W3,W4})
\begin{eqnarray}\label{E-L}
\Delta^2u&=&-\sum_{i=n+1}^L\big(\Delta\langle\nabla u,(d\nu_i\circ u)\nabla u\rangle+\nabla\cdot\langle\Delta u,(d\nu_i\circ u)\nabla u\rangle \nonumber
\\&&+\langle\nabla\Delta u,(d\nu_i\circ u)\nabla u\rangle\big)\nu_i\circ u \nonumber
\\&:=&-f(u),
\end{eqnarray}
where $\{\nu_i\}_{i=n+1}^L$ is a smooth local orthonormal frame field of the normal space of $N$. It is easy to see that
\begin{eqnarray}
|f(u)|\leq C(|\nabla^3u||\nabla u|+|\nabla^2u|^2+|\nabla u|^4).
\end{eqnarray}

Regularity issues for extrinsic biharmonic maps in dimensions $\geq4$ were first studied by Chang etc. in \cite{CWY} and for intrinsic biharmonic maps in dimension 4 by Ku \cite{Ku} and alternative proofs by Wang \cite{W1} and Strzelecki \cite{Strz} when the target manifold are the standard spheres $\mathbb{S}^n$. Wang extended the regularity result by \cite{CWY} on biharmonic maps  for general targets manifolds $N$ in \cite{W2,W3}, where he used a Coulomb gauge frame and Riesz potentials or Lorentz space estimates to prove that every weakly biharmonic maps from $\R^4$ to $N$ is smooth and every stationary biharmonic map from $\R^m(m\geq5)$ to $N$ satisfies ${\rm {dim}}\mathcal{S}\leq m-4$, i.e., the Hausdorff dimension of singular set is
at most $m-4$. Wang's partial regularity result was reproved by Lamm and Rivi\'ere \cite{Lamm2} and Struwe \cite{struwe2} extending the lower order gauge theory technique developed in \cite{Riviere1, Riviere2}.  See also Scheven \cite{Scheven} for partial regularity result for minimizing extrinsic biharmonic maps and Breiner and Lamm \cite{BL} for recent development and references therein.

The negative gradient flow for extrinsic biharmonic maps from a closed manifold (compact without boundary) was first studied by Lamm \cite{Lamm1}, where he proved the long time existence of global smooth solution  when either the dimension of $M$ is at most 3 or under a small initial energy condition in dimension 4. In general,  a finite time singularity may develop in dimension 4 \cite{Cooper,LY-2}. Motivated by the heat flow of harmonic maps from surfaces by Struwe \cite{struwe1}, it is natural to consider whether an extrinsic biharmonic map heat flow in dimension 4 has a global weak solution, which is regular outside at most finite many singularities. In this direction Gastel \cite{Gastel} and Wang \cite{W4} independently established a global weak solution for extrinsic biharmonic map heat flow in dimension 4, which is singular at most at finite time slices, but the problem of at most finite many singularities remains open (cf. Remark 1.2 of \cite{W4}).

In this paper we will study the extrinsic biharmonic map heat flow from a 4-dimensional compact manifold with boundary, i.e., we consider a solution $u\in C^{4+\alpha}(M\times(0,T),N)$ of
\begin{eqnarray}
\partial_tu+\Delta^2u&=&-f(u)\label{Heat-Flow}\label{equ1}
\\u(\cdot,0)&=&u_0,
\\u|_{\partial M}&=&g,
\\ \partial_\nu u|_{\partial M}&=&h,\label{equ2}
\end{eqnarray}
where $u_0\in W^{2,2}(M,N)$, $g\in C^{4+\alpha}(\partial M,N)$, and $h\in C^{3+\alpha}(\partial M, T_gN)$.

For $x_0\in \overline M$, let $B^M_R(x_0)$ denote the closed geodesic ball in $\overline M$ with center $x_0$ and
radius $R>0$, and set
$$E(u(t);B^M_R(x_0)):=\int_{B^M_R(x_0)}|\nabla^2u(t)|^2dx+\big(\int_{B^M_R(x_0)}|\nabla u(t)|^4dx\big)^\frac{1}{2},$$
for $0<R<\frac{1}{4}{\rm{inj}}_M$, here ${\rm{inj}}_M$ denotes the injectivity radius of $\overline M$.

The main result of this work is:

\begin{thm}\label{main thm}
For $dimM=4$, given any maps $u_0\in W^{2,2}(M,N), g\in C^{4+\alpha}(\partial M,N)$,
and $h\in C^{3+\alpha}(\partial M, T_gN)$,  there exists a unique global weak solution $u\in L^\infty(\R_+,
W^{2,2}(M, N))$ of \eqref{equ1}-\eqref{equ2}, with $u_t\in L^2(M\times\R_+)$,  satisfying:
\\(1) For any $0<T<\infty$,
\begin{eqnarray}\label{1}
2\int_0^T\int_M|u_t|^2dv_gdt+F_2(u(T))\leq F_2(u_0),
\end{eqnarray}
and $F_2(u(\cdot,t))$ is monotonically non-increasing with respect to $t\geq0$.
\\(2) There exist a positive integer $K$ depending only on $u_0, g, h, M, N$, and $0<T_1<\cdot\cdot\cdot<T_K\leq\infty$, which is characterized by the condition
\begin{align}
\limsup_{t\uparrow T_k}\max_{x\in\overline M} E(u(t);B^M_R(x))>\epsilon_1\quad \mbox{for all}\quad R>0,
\end{align}
where $\epsilon_1>0$ is the constant given by Theorem \ref{thm:C^k} below,
such that $u\in C_{loc}^{4+\alpha,1+\frac{\alpha}{4}}\big(M\times(\R_+\setminus{\cup_{k=1}^K\{T_k\}}), N\big)$.
\\(3) For each $k\in\{1,\cdots, K\}$, there exist sequences $t_{i}^k\uparrow T_k,$ $x_i^k\to x^k\in\overline M$,
and  $r_{i}^k\to0$ such that
\begin{itemize}
\item[(i)] if $x^k\in M$, there exists a non-constant biharmonic map $\omega^k\in C^\infty\cap W^{2,2}(\R^4,N)$ such that
    \begin{eqnarray}\label{3}
u_{i}^k(x)=u(x_i^k+r_{i}^kx, t_{i}^k)\to \omega^k ~{\rm{in}} ~C^4_{loc}(\R^4).
\end{eqnarray}
\item[(ii)] if $x^k\in \partial M$ and if $\displaystyle\limsup_{i\to\infty}\frac{dist(x_i^k,\partial M)}{r_i^k}\to \infty,$ then statement $(i)$ holds.
If there exists $0\le a<+\infty$ such that $\displaystyle\limsup_{i\to\infty}\frac{dist(x_i^k,\partial M)}{r_i^k}=a,$
 then there exists a non-constant biharmonic map $\omega^k\in C^\infty\cap W^{2,2}(\R_a^4,N)$, 
 with $\omega={\rm{constant}}, \ \partial_\nu\omega=0$ on $\partial \R_a^4$, such that
    \begin{eqnarray}
u_{i}^k(x)=u(x_i^k+r_{i}^kx, t_{i})\to \omega^k~{\rm{in}} ~C^4_{loc}(\R_a^{4+}),
\end{eqnarray}
where $\R_a^4:=\big\{(x^1,x^2,x^3,x^4)\in\R^4\ |\ x^4\geq-a\big\}$
and $\R_a^{4+}:=\big\{(x^1,x^2,x^3,x^4)\in\R^4\ |\ x^4>-a\big\}$.
\end{itemize}
\end{thm}

As an application of the heat flow of biharmonic maps, we obtain the following existence result.
\begin{thm}\label{thm:main-02}
Let $u$ be the global solution of \eqref{equ1}-\eqref{equ2} obtained by Theorem \ref{main thm}.
Then there exists $t_i \uparrow \infty$ such that $u(\cdot,t_i)$ converges weakly in $W^{2,2}(M)$ to a biharmonic map $u_\infty\in C^{4+\alpha}(\overline M,N)$ with boundary data $u_\infty|_{\partial M}=g$ and $\partial_\nu u_\infty|_{\partial M}=h$.
\end{thm}

The paper is organized as follows. In section 2, we prove a small energy regularity result for biharmonic maps, a
main tool in our a priori estimates, and as a corollary, we obtain a gap theorem for biharmonic map under the Dirichlet boundary condition. At the end of this section, we prove several interpolation inequalities which will be used frequently in the subsequent sections. In section 3, we give a priori estimates for the heat flow and
the uniform local $W^{4,2}$ estimates in time under the assumption of small energy on a ball.
In section 4, we prove the main theorems, Theorem \ref{main thm} and Theorem \ref{thm:main-02}.

Throughout this paper, the letter $C$ denotes a positive constant that depends only on $M,N,u_0,g$, whose values
may vary from lines to lines. If it depends on some other quantity, then we will point it out. For example, $C(R)$ is a positive constant depends on $R$.

\

\noindent\textbf{Additional Notations}. For $\Omega\subset\R^4$ and $0\le s<t\le\infty$,
denote $\Omega_s^t=\Omega\times [s,t]$, $M_s^t=M\times [s,t]$, and $M^T=M\times [0,T]$.
Also denote the standard Sobolev and H\"{o}lder spaces by
$W^{m,n}_p(M^T)$ and $C^{m+\alpha,n+\beta}(M^T)$.

We denote $B_R$ (or $B_R(0)$) as the standard ball in $\R^4$ with radius $R$ and center $0$. Denote $x'=(x^1,x^2,x^3)\in\R^3$,
\begin{align*}
B^+_R:=\Big\{(x',x^4)||x'|^2+|x^4|^2\leq R^2,\ x^4\geq 0\Big\}\ and \ \partial^0B^+_R:=\Big\{(x',x^4)||x'|^2+|x^4|^2\leq R^2,\ x^4= 0\Big\},
\end{align*}
and
\begin{align*}
V(M_s^t):=\Big\{u&:M\times [s,t]\to N|\sup_{s\leq\sigma\leq t}\big(\|\nabla^2u\|_{L^2(M)}+\|\nabla u\|_{L^{4}(M)}\big)\\
&\quad+\int_{M^t_s}(|\partial_tu|^2+|\nabla^4u|^2)\,dv_gdt<\infty\Big\}.
\end{align*}

\section{Some basic theorems and interpolation inequalities}

In this section we prove several basic theorems, including the small energy regularity theorem and the gap theorem. At the end of this section, we derive some interpolation inequalities which will be used later.

\begin{thm}($\varepsilon_1-$regularity)\label{thm:small-regularity}\\
(i) If $u\in W^{4,p}(B_1)$, $p>1$, is an approximated biharmonic map with bi-tension field $\tau_2(u)\in L^p(B_1)$, i.e.$$\Delta^2u=-f(u)+\tau_2(u),$$ where $f(u)$ is defined in \eqref{E-L}.
Then there exists a constant $\epsilon_1>0$ such that if $E(u;B_1)\leq\epsilon_1$, then
\begin{eqnarray*}
\big\|u-\bar{u}\big\|_{W^{4,p}(B_{1/2})}\leq C(p,N)\Big(\|\nabla^2u\|_{L^2(B_1)}+\|\nabla u\|_{L^4(B_1)}+\|\tau_2 (u)\|_{L^p(B_1)}\Big),
\end{eqnarray*}
where $\displaystyle\bar{u}=\frac{1}{|B_1|}\int_{B_1}udx$ is the mean value of $u$ over the unit ball.\\
(ii) If $u\in W^{4,p}(B^+_1)$, $p>1$, is an approximated biharmonic map with tension field $\tau_2(u)\in L^p(B^+_1)$ and the Dirichlet boundary value $$u|_{\partial^0B^+_1}=g\quad and \quad \frac{\partial u}{\partial \vec{n}}|_{\partial^0B^+_1}=h,$$ where $g\in C^4(\partial^0B_1^+), h\in C^3(\partial^0 B_1^+)$ and $\vec{n}$ is the outward unit normal vector of $\partial^0B^+_1$. Then there exists a constant $\epsilon_1>0$ such that if $E(u;B^+_1)\leq\epsilon_1$, then
\begin{eqnarray*}
&&\|u-\bar{u}\|_{W^{4,p}(B_{1/2}^+)}\\
&&\leq C(p,N)\Big(\|\nabla^2u\|_{L^2(B_1^+)}+\|\nabla u\|_{L^4(B_1^+)}+\|\tau_2 (u)\|_{L^p(B_1^+)}+\|g\|_{W^{4,2}(\partial^0 B_1^+)}+\|h\|_{W^{3,2}(\partial^0B_1^+)}\Big),
\end{eqnarray*}
where $\displaystyle\bar{u}:=\frac{1}{|\partial^0 B^+_{1/2}|}\int_{\partial^0 B^+_{1/2}}u$ is the mean value of $u$ over the boundary $\partial^0B_1^+$.
\end{thm}
\begin{proof}
Here we use the idea of \cite{LY} to give the proof of boundary estimate stated in (ii),
and leave the interior estimate in (i) for interested readers since it is similar to (ii)  and easier to obtain.

For convenience, assume $\bar{u}=0$. Since $u$ satisfies the Euler-Lagrange equation:
\begin{eqnarray*}
\triangle^2u=\nabla^3u\#\nabla u+\nabla^2 u\#\nabla^2 u+\nabla^2 u\#\nabla u\#\nabla u+\nabla u\#\nabla u\#\nabla u\#\nabla u+\tau_2(u).
\end{eqnarray*}
Here $\#$ denotes some 'product' for which we are only interested in the properties such as
\begin{equation*}
	|a\# b|\leq C |a| |b|.
\end{equation*}
For $0<\sigma<1$ and $\sigma'=\frac{1+\sigma}{2}$, let $\varphi\in C_0^\infty(B^+_{\sigma'})$ be a cut-off function, satisfying $\varphi\equiv1$ in $B^+_\sigma$ and $|\nabla^j\varphi|\leq\frac{4^j}{(1-\sigma)^j}$
for $j=1,2,3,4$.
Direct computations show that
\begin{align*}
\triangle^2(\varphi u)&=\triangle (\varphi\triangle u+2\nabla u \nabla\varphi+u\triangle \varphi)\\
&=
\varphi\triangle ^2u+4\nabla\triangle u\nabla\varphi+2\triangle u\triangle \varphi+4\nabla^2u\nabla^2\varphi+4\nabla u\nabla\triangle \varphi+u\triangle ^2\varphi\\
&=
(\nabla^3u\#\nabla u+\nabla^2 u\#\nabla^2 u+\nabla^2 u\#\nabla u\#\nabla u+\nabla u\#\nabla u\#\nabla u\#\nabla u+\tau_2(u))\varphi\\
&\quad+\nabla^3 u\#\nabla\varphi+\nabla^2u\#\nabla^2\varphi+\nabla u\#\nabla^3\varphi+u\nabla^4\varphi\\
&=
(\nabla^3(\varphi u)\#\nabla u+\nabla^2 (\varphi u)\#\nabla^2 u+\nabla^2 u\#\nabla u\#\nabla (\varphi u)+\nabla u\#\nabla u\#\nabla u\#\nabla (\varphi u))\\
&\quad+\nabla^3 u\#\nabla\varphi+\nabla^2u\#\nabla^2\varphi+\nabla u\#\nabla^3\varphi+u\nabla^4\varphi
+\nabla^2u\#\nabla u\#\nabla\varphi+\nabla^2\varphi\#\nabla u\#\nabla u\\
&\quad+\nabla u\#\nabla u\#\nabla u\#\nabla\varphi+\varphi\tau_2(u).
\end{align*}
Assume first that $1<p<\frac{4}{3}$. Observe that
$$\varphi u=\varphi g, \ \frac{\partial(\varphi u)}{\partial \vec{n}}=\varphi h+\frac{\partial\varphi}{\partial\vec{n}}g
\ \ {\rm{on}}\ \ \partial^0 B_1^+.$$
By the standard $L^p$ theory (cf. \cite{Hong-Yin}), we have
\begin{eqnarray*}
&&\|\nabla^4(\varphi u)\|_{L^p(B_1^+)}\le\\
&& C\Big(\|\nabla u\|_{L^4(B_1^+)}\|\nabla^3(\varphi u)\|_{L^{\frac{4p}{4-p}}(B_1^+)}
+\|\nabla^2 u\|_{L^2(B_1^+)}\|\nabla^2(\varphi u)\|_{L^{\frac{4p}{4-2p}}(B_1^+)}\\
&&+\|\nabla^2 u\|_{L^2(B_1^+)}\|\nabla u\|_{L^4(B_1^+)}\|\nabla(\varphi u)\|_{L^{\frac{4p}{4-3p}}(B_1^+)}
+\|\nabla u\|^3_{L^4(B_1^+)}\|\nabla(\varphi u)\|_{L^{\frac{4p}{4-3p}}(B_1^+)} \\
&&+\frac{\|\nabla^3 u\|_{L^p(B_{\sigma'}^+)}}{1-\sigma}+\frac{\|\nabla^2 u\|_{L^p(B_{\sigma'}^+)}}{(1-\sigma)^2}
+\frac{\|\nabla u\|_{L^p(B_{\sigma'}^+)}}{(1-\sigma)^3}\\
&&+\frac{\|u\|_{L^p(B_{\sigma'}^+)}}{(1-\sigma)^4}
+\frac{\|\nabla^2 u\#\nabla u\|_{L^p(B_{\sigma'}^+)}}{1-\sigma}+\frac{\|\nabla u\#\nabla u\|_{L^p(B_{\sigma'}^+)}}{(1-\sigma)^2}\\
&&+\frac{1}{1-\sigma}\|\nabla u\#\nabla u\#\nabla u\|_{L^p(B_{\sigma'}^+)}+\|\varphi\tau_2(u)\|_{L^p(B_1^+)}\\
&&+\|\varphi g\|_{W^{4,p}(\partial^0 B_1^+)}+\|\varphi h+\frac{\partial\varphi}{\partial\vec{n}} g\|_{W^{3,p}(\partial^0
B_1^+)} \Big).
\end{eqnarray*}
By the Sobolev embedding, if $\epsilon_1$ is chosen to be sufficiently small, then we get
\begin{eqnarray*}
&&\|\nabla^4(\varphi u)\|_{L^p(B_1^+)}\le\\
&& C\Big(
\frac{1}{1-\sigma}\|\nabla^3 u\|_{L^p(B_{\sigma'}^+)}+\frac{1}{(1-\sigma)^2}\|\nabla^2 u\|_{L^p(B_{\sigma'}^+)}
+\frac{1}{(1-\sigma)^3}\|\nabla u\|_{L^p(B_{\sigma'}^+)}\\
&&+\frac{1}{(1-\sigma)^4}\|u\|_{L^p(B_{\sigma'}^+)}
+\frac{1}{1-\sigma}\|\nabla^2 u\#\nabla u\|_{L^p(B_{\sigma'}^+)}+\frac{1}{(1-\sigma)^2}\|\nabla u\#\nabla u\|_{L^p(B_{\sigma'}^+)}\\
&&+\frac{1}{1-\sigma}\|\nabla u\#\nabla u\#\nabla u\|_{L^p(B_{\sigma'}^+)}+\|\varphi\tau_2(u)\|_{L^p(B_1^+)}\\
&&+\|\varphi g\|_{W^{4,p}(\partial^0B_1^+)}+\|\varphi h\|_{W^{3,p}(\partial^0 B_1^+)}
+\|\frac{\partial\varphi}{\partial\vec{n}} g\|_{W^{3,p}(\partial^0
B_1^+)}\Big).
\end{eqnarray*}
Setting
\[
\Psi_j(p)=\sup_{0\leq\sigma\leq1}(1-\sigma)^j\|\nabla^j u\|_{L^p(B^+_{\sigma})},
\]
and noticing that $1-\sigma=2(1-\sigma')$, $1<p<\frac{4}{3}$, we have
\begin{align*}
&\Psi_4(p)\\
&\leq C\Big(\sum_{j=0}^3\Psi_j(p)+\|\nabla^2 u\#\nabla u\|_{L^p(B_1^+)}+\|\nabla u\#\nabla u\|_{L^p(B_1^+)}\\ &\quad +\|\nabla u\#\nabla u\#\nabla u\|_{L^p(B_1^+)} +\|\varphi\tau_2(u)\|_{L^p(B_1^+)}\\
&\quad+\sup_{0\leq\sigma\leq1}(1-\sigma)^4\big[
\|\varphi g\|_{W^{4,p}(\partial^0B_1^+)}+\|\varphi h\|_{W^{3,p}(\partial^0B_1^+)}
+\|\frac{\partial\varphi}{\partial\vec{n}} g\|_{W^{3,p}(\partial^0 B_1^+)}\big]\Big)\\
&\leq
C\Big(\sum_{j=1}^3\Psi_j(p)+
\|\nabla^2u\|_{L^2(B_1^+)}+\|\nabla u\|_{L^4(B_1^+)}+\|\tau_2(u)\|_{L^p(B_1^+)}\\
&\quad+\|g\|_{W^{4,p}(\partial^0B_1^+)}+\|h\|_{W^{3,p}(\partial^0 B_1^+)}\Big).
\end{align*}
Using the interpolation inequality (see \cite{LY})
\begin{equation*}
	\Psi_j(p)\leq \epsilon^{4-j} \Psi_4(p) +C\epsilon^{-j} \Psi_0(p), \ j=1,2,3,\ \epsilon>0,
\end{equation*}
we get, by choosing sufficiently small $\epsilon>0$,
\begin{eqnarray}\label{4p-estimate}
\Psi_4(p)&\leq&
C\Big(\Psi_0(p)+\|\nabla^2u\|_{L^2(B_1^+)}+\|\nabla u\|_{L^4(B_1^+)}+\|\tau_2(u)\|_{L^p(B_1^+)}\nonumber\\
&&\quad+\| g\|_{W^{4,2}(\partial^0B_1^+)}+\|h\|_{W^{3,2}(\partial^0 B_1^+)}\Big)\nonumber\\
&\leq&
C\Big(\|\nabla^2u\|_{L^2(B_1^+)}+\|\nabla u\|_{L^4(B_1^+)}+\|\tau_2(u)\|_{L^p(B_1^+)}\nonumber\\
&&\quad+\| g\|_{W^{4,2}(\partial^0B_1^+)}+\|h\|_{W^{3,2}(\partial^0 B_1^+)}\Big),
\end{eqnarray}
where we have used the Poincar\'e inequality in the last step.

If $p\geq \frac{4}{3}$, we start by applying (\ref{4p-estimate}) with $p=\frac{16}{13}$ so that
\begin{eqnarray*}
	\|u\|_{W^{4,\frac{16}{13}}(B^+_{7/8})}&\leq& C \Big(\|\nabla^2 u\|_{L^2(B^+_1)} +\|\nabla u\|_{L^4(B^+_1)}+\|\tau_2(u)\|_{L^p(B_1^+)}\\
&&+\|g\|_{W^{4,2}(\partial^0B_1^+)}+\|h\|_{W^{3,2}(\partial^0 B_1^+)}\Big).
\end{eqnarray*}
This, combined with the Sobolev embedding theorem, implies that
\begin{align*}
	&\|\nabla^3 u\|_{L^{\frac{16}{9}}(B^+_{7/8})} + \|\nabla^2 u\|_{L^{\frac{16}{5}}(B^+_{7/8})} +\|\nabla u\|_{L^{16}(B^+_{7/8})}\\
&\leq C \Big(\|\nabla^2 u\|_{L^2(B^+_1)} +\|\nabla u\|_{L^4(B^+_1)}+\|\tau_2(u)\|_{L^p(B_1^+)}
+\|g\|_{W^{4,2}(\partial^0B_1^+)}+\|h\|_{W^{3,2}(\partial^0B_1^+)}\Big).
\end{align*}
With this estimate, we can bound the $L^{\min\{\frac{8}{5},p\}}$-norm of the right hand side of the Euler-Lagrange equation of $u$. The interior $L^p$-estimate together (\ref{4p-estimate}) show
that $u$ is bounded in $W^{4,{\min\{\frac{8}{5},p\}}}(B^+_{3/4})$. The lemma can be
finally proved by applying the standard bootstrapping method.
\end{proof}

As a direct corollary of the above theorem, we can get the following gap theorem.
\begin{thm}[Gap-phenomena]\label{thm:Gap-theorem}
Suppose either $u\in C^\infty(\R^4,N)$ is a biharmonic map or $u\in C^\infty(\R_+^4,N)$ is a biharmonic map with the Dirichlet boundary condition:
$$u|_{\partial \R_+^4}={\rm{constant}}\quad {\rm{and}} \quad \frac{\partial u}{\partial\vec{n}}|_{\partial \R_+^4}=0.$$
Then there exists a universal constant $\epsilon_0>0$ such that if either
$$\int_{\R^4}|\Delta u|^2dx\leq \epsilon_0^2\ {\rm{or}} \ \int_{\R_+^4}|\Delta u|^2dx\leq \epsilon_0^2,$$ then $u$ is a constant map.
\end{thm}
\begin{proof}
For simplicity, we only prove the upper half space case, since the proof of $u\in C^\infty(\R^4,N)$ is similar.
By Poincar\'e's inequality and integration by parts, we have that for any $R>0$, it holds
$$\frac{1}{4R^2}\int_{B_{2R}^+}|\nabla u|^2\,dx\leq C\int_{B_{2R}^+}|\nabla^2 u|^2\,dx
\leq C\int_{\R^4_+}|\nabla^2 u|^2\,dx= C\int_{\R^4_+}|\Delta u|^2\,dx
\leq C\epsilon_0^2.$$
Hence, by the standard elliptic estimates and Sobolev's embedding, we have
$$\int_{B_R^+}|\nabla^2 u|^2\,dx+\big(\int_{B_R^+}|\nabla u|^4\,dx\big)^{\frac{1}{2}}
\leq C\big[\int_{B_{2R}^+}|\Delta u|^2\,dx+\frac{1}{R^2}\int_{B_{2R}^+}|\nabla u|^2\,dx\big]
\leq C\epsilon_0^2.$$
Choosing $\epsilon_0<<\epsilon_1$ and applying both Theorem \ref{thm:small-regularity} and
the Sobolev embedding, we have that for any $R>0$, there holds
\begin{align*}
R\|\nabla u\|_{L^\infty(B^+_R)}\leq C\big(\|\nabla^2 u\|_{L^2(B^+_{2R})} +\|\nabla u\|_{L^4(B^+_{2R})}\big)
\leq C.
\end{align*}
Sending $R$ to infinity yields that $u$ is a constant map.
\end{proof}

In the following we will prove several interpolation type inequalities,
which will be used through the remaining sections.

\begin{lem}\label{Lem:interpolation}
For any $u\in W^{4,2}(M,N)$, we have
\begin{align}
\int_{B^M_R}|\nabla^3u|^2\,dx&\leq C R^2\int_{B^M_R}|\nabla^4u|^2\,dx
+\frac{C}{ R^2}\int_{B^M_R}|\nabla^2u|^2\,dx,\label{ine4}
\\
\big(\int_{B^M_R}|\nabla^3u|^4\,dx\big)^\frac{1}{2}
&\leq C\big(\int_{B^M_R}|\nabla^4u|^2\,dx+\frac{C}{R^4}\int_{B^M_R}|\nabla^2u|^2\,dx\big),\label{ine5}
\\
\int_{B^M_R}|\nabla^2u|^4\,dx&\leq C\int_{B^M_R}|\nabla^2u|^2\,dx
\big(\int_{B^M_R}|\nabla^4u|^2\,dx+\frac{1}{R^4}\int_{B^M_R}|\nabla^2u|^2\,dx\big),\label{ine6}
\\
\int_{B^M_R}|\nabla u|^8\,dx&\leq C\int_{B^M_R}|\nabla u|^4\,dx\big[\int_{B^M_R}|\nabla^2u|^2\,dx
\big(\int_{B^M_R}|\nabla^4u|^2\,dx+\frac{1}{R^4}\int_{B^M_R}|\nabla^2u|^2\,dx\big)\notag\\
&\quad+\frac{1}{R^4}\int_{B^M_R}|\nabla u|^4\,dx\big].\label{ine7}
\end{align}
\end{lem}
\proof \eqref{ine4} is a standard interpolation inequality (cf. \cite{GT}, page 173).
By the Sobolev embedding $W^{1,2}\hookrightarrow L^4$ on $B^M_R$  we get
$$\big(\int_{B^M_R}|\nabla^3u|^4\,dx\big)^{\frac{1}{2}}
\leq C\big(\int_{B^M_R}|\nabla^4u|^2\,dx+\frac{1}{R^2}\int_{B^M_R}|\nabla^3u|^2\,dx\big),$$
then \eqref{ine5} is a consequence of \eqref{ine4}.
By Sobolev embedding $W^{1,\frac{4}{3}}\hookrightarrow L^2$, we have
\begin{align*}
\int_{B^M_R}|\nabla^2u|^4\,dx &\leq \frac{C}{R^2}\||\nabla^2u|^2\|^2_{L^{\frac{4}{3}}(B^M_R)}
+C\|\nabla^2u\#\nabla^3u\|^2_{L^{\frac{4}{3}}(B^M_R)}\\
&\leq \frac{C}{R^2}\|\nabla^2u\|^2_{L^{2}(B^M_R)}\|\nabla^2u\|^2_{L^{4}(B^M_R)}
+C\|\nabla^2u\|^2_{L^{2}(B^M_R)}\|\nabla^3u\|^2_{L^{4}(B^M_R)}\\
&\leq \frac{1}{2}\int_{B^M_R}|\nabla^2u|^4\,dx+ \frac{C}{R^4}\|\nabla^2u\|^4_{L^{2}(B^M_R)}
+C\|\nabla^2u\|^2_{L^{2}(B^M_R)}\|\nabla^3u\|^2_{L^{4}(B^M_R)}\\
&\leq \frac{1}{2}\int_{B^M_R}|\nabla^2u|^4\,dx
+ C\int_{B^M_R}|\nabla^2u|^2\,dx\big[(\int_{B^M_R}|\nabla^3u|^4\,dx)^\frac{1}{2}
+\frac{1}{R^4}\int_{B^M_R}|\nabla^2u|^2\,dx\big]\\
&\leq \frac{1}{2}\int_{B^M_R}|\nabla^2u|^4\,dx+ C\int_{B^M_R}|\nabla^2u|^2\,dx
\big(\int_{B^M_R}|\nabla^4u|^2\,dx+\frac{1}{R^4}\int_{B^M_R}|\nabla^2u|^2\,dx\big),
\end{align*}
which implies \eqref{ine6}.

By Sobolev embedding $W^{1,2}\hookrightarrow L^4$, we have
\begin{align*}
\int_{B^M_R}|\nabla u|^8\,dx
&\leq C\frac{1}{R^4}\big(\int_{B^M_R}|\nabla u|^4\,dx\big)^2+C\big(\int_{B^M_R}|\nabla u|^2|\nabla^2 u|^2\,dx\big)^2
\\&\leq C\int_{B^M_R}|\nabla u|^4\,dx\big(\int_{B^M_R}|\nabla^2u|^4\,dx+\frac{1}{R^4}\int_{B^M_R}|\nabla u|^4\,dx\big)
\\&\leq C\int_{B^M_R}|\nabla u|^4\,dx\big(\int_{B^M_R}|\nabla^2u|^2\,dx
(\int_{B^M_R}|\nabla^4u|^2\,dx+\frac{1}{R^4}\int_{B^M_R}|\nabla^2u|^2\,dx\big)\\
&\quad+\frac{1}{R^4}\int_{B^M_R}|\nabla u|^4\,dx\big),
\end{align*}
which implies \eqref{ine7}, here in the last inequality we used \eqref{ine6}.
\endproof
\section{A priori estimates}

In this section we will show some properties of the flow and some a priori estimates, including the monotonicity of the energy $F_2(u)$ and small energy regularity theorem of parabolic case, which will be needed in the next section for the existence result.

From now on, we will use $\eta$ as a smooth cut off function satisfying the following properties:
\begin{eqnarray}\label{inecf}
&&\eta \in C^\infty(M),\ 0\leq\eta\leq 1,\ \eta\equiv1 ~on~B^M_R(x_0), ~\eta\equiv0 ~{\rm{on}}~M\setminus B^M_{2R}(x_0),\nonumber
\\&&\|\nabla^j\eta\|_{L^\infty}\leq\frac{C}{R^j} \ (j=1,2),\label{ine3}
\end{eqnarray}
where $x_0\in M$ and $0<R<\frac{1}{4}{\rm{inj}}_M$.

\begin{lem}\label{lem:01}
Let $u\in V(M^T)$ be a solution of \eqref{equ1}-\eqref{equ2}. Then for all $t\in[0,T)$, we have
\begin{align}
F_2(u(t))+2\int_{M^t}|\partial_tu|^2\,dv_gdt&=F_2(u_0),\label{ine1}
\\
\int_M|\nabla^2u|^2\,dv_g+\big(\int_M|\nabla u|^4\,dv_g\big)^\frac{1}{2}
&\leq C\big(F_2(u_0)+\|g\|^2_{W^{2,2}(M,N)}\big). \label{ine2}
\end{align}
Moreover, $F_2(u(t))$ is absolutely continuous in $[0,T)$ and monotonically non-increasing.
\end{lem}
\proof Multiplying the equation \eqref{equ1} by $\partial_t u$ and integrating by parts, we have
\begin{align*}
0&=\int_{M^t}|\partial_tu|^2\,dv_gdt+\int_{M^t}\Delta^2 u\partial_tu\,dv_gdt\\
&=\int_{M^t}|\partial_tu|^2\,dv_gdt+\int_{M^t}\Delta u\partial_t\Delta u\,dv_gdt
+\int_{0}^t\int_{\partial M}\partial_\nu\Delta u\partial_t u-\int_{0}^t\int_{\partial M}\Delta u\partial_t\partial_\nu u\\
&=\int_{M^t}|\partial_tu|^2\,dv_gdt+\int_{M^t}\partial_t(\frac{1}{2}|\Delta u|^2)\,dv_gdt,
\end{align*}
where we used $\partial_t u|_{\partial M}=\partial_t\partial_\nu u|_{\partial M}=0$.
Hence \eqref{ine1} follows immediately. Moreover, it is easy to see $F_2(u(t))$ is absolutely continuous in $[0,T]$ and monotonically non-increasing.

For \eqref{ine2}, we first use the $L^2$-estimate for the Laplace operator $\Delta$ to get
$$\int_M|\nabla^2u|^2\,dv_g\leq C\big(F_2(u(t))+\|g\|^2_{W^{2,2}(M,N)}\big)
\leq C\big(F_2(u_0)+\|g\|^2_{W^{2,2}(M,N)}\big).$$
Then, by Sobolev's inequality we have
$$\int_M|\nabla(u-g)|^2\,dv_g\leq C\int_M|\nabla^2(u-g)|^2dv_g,$$
and hence
$$\int_M|\nabla u|^2\,dv_g\leq C\big(\int_M|\nabla^2u|^2\,dv_g+\|g\|^2_{W^{2,2}(M,N)}\big).$$
Observe that (\ref{ine2}) is a consequence of the following Sobolev inequality
\begin{eqnarray*}
\big(\int_M|\nabla u|^4\,dv_g\big)^\frac{1}{2}\leq C\big(\int_M|\nabla^2u|^2\,dv_g
+\int_M|\nabla u|^2\,dv_g\big).
\end{eqnarray*}
This completes the proof. \endproof

With the help of Theorem \ref{thm:small-regularity}, we have
\begin{lem}\label{lem:02}
There exists $\epsilon_1>0$ such that if $u\in V(M^T)$ is a solution of \eqref{equ1}-\eqref{equ2}
satisfying $E(u(t);B^M_{2R}(x_0))\leq\epsilon_1$ for some $R>0$, then we have
\begin{eqnarray}\label{ine8}
\int_{B^M_R(x_0)}|\nabla^4u|^2\,dx+\frac{1}{R^2}\int_{B^M_R(x_0)}|\nabla^3u|^2\,dx\leq C\int_{B^M_{2R}(x_0)}|\partial_tu|^2\,dx+\frac{C}{R^4}.
\end{eqnarray}
\end{lem}
\proof Since $u$ satisfies \eqref{Heat-Flow}, we have that\\
(i) if $B^M_{2R}(x_0)\cap \partial M=\emptyset$, then by taking $\tau_2(u)=\partial_tu$ in Theorem \ref{thm:small-regularity} (i) and applying a standard scaling argument, we have
\begin{align*}
\int_{B^M_R(x_0)}|\nabla^4u|^2dx+\frac{1}{R^2}\int_{B^M_R(x_0)}|\nabla^3u|^2dx&\leq C\int_{B^M_{2R}(x_0)}|\partial_tu|^2dx+\frac{CE(u(t);B^M_{2R}(x_0))}{R^4}\\
&\leq C\int_{B^M_{2R}(x_0)}|\partial_tu|^2dx+\frac{C}{R^4},
\end{align*}(ii)
if $B^M_{2R}(x_0)\cap \partial M\neq\emptyset$, then Theorem \ref{thm:small-regularity} (ii) implies
that
\begin{align*}
&\int_{B^M_R(x_0)}|\nabla^4u|^2dx+\frac{1}{R^2}\int_{B^M_R(x_0)}|\nabla^3u|^2dx\\
&\leq C\int_{B^M_{2R}(x_0)}|\partial_tu|^2dx+C\frac{E(u(t);B^M_{2R}(x_0))
+\|g\|^2_{W^{4,2}(\partial^0 B^M_{2R}(x_0))}+\|h\|^2_{W^{3,2}(\partial^0 B^M_{2R}(x_0))}}{R^4}\\
&\leq C\int_{B^M_{2R}(x_0)}|\partial_tu|^2dx+\frac{C}{R^4}.
\end{align*}
Here $\partial^0B^M_{2R}(x_0)=\partial B^M_{2R}(x_0)\cap\partial M$.
Hence the conclusion of the lemma follows.\endproof

From Lemma \ref{lem:01} and Lemma \ref{lem:02}, we can easily obtain the following corollary.
\begin{cor}\label{cor1}
Let $u\in V(M^T)$ be a solution of \eqref{equ1}-\eqref{equ2}. Assume that there exists $R>0$ such that $$\sup_{0\leq t<T}E(u(t);B^M_{2R}(x_0))\leq\epsilon_1.$$ Then we have for all $t\in [0,T)$
\begin{eqnarray}
&&\int_{(B^M_{R}(x_0))^t}|\nabla^3u|^2\leq C+\frac{ct}{R^2},\label{ine9}
\\&&\int_{(B^M_{R}(x_0))^t}|\nabla^4u|^2\leq C+\frac{ct}{R^4}.\label{ine10}
\end{eqnarray}
\end{cor}
\begin{proof}
Integrating \eqref{ine8} from $0$ to $t$ and applying Lemma \ref{lem:01} yields \eqref{ine9} and \eqref{ine10}.
\end{proof}

In the next step we derive an $L^2$-estimate for $\partial_tu$, which in turn yields an $L^2$-estimate
for $\nabla^4u$, and then we can apply both $L^p$ and Schauder estimates to achieve the desired $C^l$ estimates.

\begin{lem}\label{lem:03}
Let $u\in V(M^T)\cap_{\sigma>0}C^4(M_\sigma^T;N)$ be a solution of \eqref{equ1}-\eqref{equ2}. Assume that there exists $R>0$ such that $$\sup_{0\leq t<T}E(u(t);B^M_{4R}(x_0))\leq\epsilon_1.$$ Then there exists $0<\delta<min\{T, CR^4\}$ such that for all $s,t\in(0,T)$ with $s<t$ and $|t-s|<\delta$, we have
\begin{eqnarray}\label{inemain}
\sup_{s\leq t'\leq t}\int_M\eta^4|\partial_tu(\cdot,t')|^2\,dx\leq C\int_M\eta^4|\partial_tu(\cdot,s)|^2\,dx+\frac{C}{R^4},
\end{eqnarray}
where $\eta$ is a cut off function, with support in $B_{2R}(x_0)$, defined as in \eqref{ine3}.
\end{lem}
\proof

Differentiating equation \eqref{Heat-Flow} with respect to $t$, multiplying the resulting equation with $\eta^4\partial_tu$, and integrating over $M$ and applying integration by parts, we get
\begin{eqnarray}\label{ine11}
&&\frac{1}{2}\int_{M_s^t}\eta^4\partial_t|\partial_tu|^2+\int_{M_s^t}\eta^4|\Delta\partial_tu|^2
+2\int_{M_s^t}\nabla\eta^4\nabla\partial_tu\Delta\partial_tu+\int_{M_s^t}\Delta\eta^4\partial_tu\Delta\partial_tu\nonumber
\\&\leq& C\int_{M_s^t}\eta^4(|\nabla\Delta u||\nabla u||\partial_tu|^2+
|\nabla^2u|^2|\partial_tu|^2+|\nabla u|^4|\partial_tu|^2)\nonumber
\\&:=&I_1+I_2+I_3.
\end{eqnarray}
Without loss of generality,  we may assume that
$$\sup_{s\leq t'\leq t}\int_M\eta^4|\partial_tu(\cdot,t')|^2=\int_M\eta^4|\partial_tu(\cdot,t)|^2.$$
Let's first estimate $I_1$. With the help of H\"older's inequality and the Sobolev embedding $W^{1,2}(M)\hookrightarrow L^4(M)$, we get
\begin{align*}
I_1&\leq C\int^t_s(\int_{B^M_{2R}(x_0)}|\nabla u|^4)^\frac{1}{4}(\int_M\eta^8|\partial_tu|^4)^\frac{1}{2}(\int_{B^M_{2R}(x_0)}|\nabla\Delta u|^4)^\frac{1}{4}
\\&\leq C\epsilon_1^\frac{1}{2}\int^t_s(\int_M\eta^4|\partial_tu|^2+|\nabla\eta|^2\eta^2|\partial_tu|^2+\eta^4|\nabla\partial_tu|^2)\big[\int_{B^M_{2R}(x_0)}|\nabla^4u|^2+\frac{1}{R^2}|\nabla^3u|^2\big]^\frac{1}{2}
\end{align*}
Since  $\partial_tu|_{\partial M}=0$,  by integration by part we get
\begin{align*}
\int_M\eta^4|\nabla\partial_tu|^2
&=-\int_M\Delta\partial_tu\partial_tu\eta^4
+4\nabla\partial_tu\partial_tu(\eta^3\nabla\eta)\\
&\leq \int_M|\Delta\partial_tu\partial_tu\eta^4|
+\frac{1}{2}\int_M\eta^4|\nabla\partial_tu|^2+C\int_M\eta^2|\nabla\eta|^2|\partial_tu|^2.
\end{align*}
Thus we have
\begin{align*}
\int_M\eta^4|\nabla\partial_tu|^2
\leq C\int_M|\Delta\partial_tu\partial_tu\eta^4|
+C\int_M\eta^2|\nabla\eta|^2|\partial_tu|^2.
\end{align*}
Therefore we obtain
\begin{align}\label{ine12}
I_1&\leq C\epsilon_1^\frac{1}{2}\int^t_s(\int_M|\Delta\partial_tu\partial_tu\eta^4|+\eta^4|\partial_tu|^2+|\nabla\eta|^2\eta^2|\partial_tu|^2)
(\int_{B^M_{2R}}|\nabla^4u|^2+\frac{1}{R^2}|\nabla^3u|^2)^\frac{1}{2}\nonumber
\\&\leq C\epsilon_1^\frac{1}{2}\int^t_s\big((\int_M\eta^4|\Delta\partial_tu|^2)^\frac{1}{2}(\int_M\eta^4|\partial_tu|^2)^\frac{1}{2}
+\int_M\eta^4|\partial_tu|^2+\int_M|\nabla\eta|^2\eta^2|\partial_tu|^2\big)\nonumber
\\&\quad \times(\int_{B^M_{2R}}|\nabla^4u|^2+\frac{1}{R^2}|\nabla^3u|^2)^\frac{1}{2}\nonumber
\\&\leq C\epsilon_1^\frac{1}{2}\int^t_s\big[(\int_M\eta^4|\Delta\partial_tu|^2)^\frac{1}{2}(\int_M\eta^4|\partial_tu|^2)^\frac{1}{2}
+\int_M\eta^4|\partial_tu|^2\notag
\\&\quad+(\int_M\eta^4\partial_tu|^2)^{\frac{1}{2}}(\int_M|\nabla\eta|^4\partial_tu|^2)^{\frac{1}{2}}\big]\nonumber
[\int_{B^M_{2R}}|\nabla^4u|^2+\frac{1}{R^2}\int_{B^M_{2R}}|\nabla^3u|^2]^\frac{1}{2}\nonumber
\\&\leq C\epsilon_1^\frac{1}{2}(\sup_{s\leq t'\leq \nonumber t}\int_M\eta^4|\partial_tu(\cdot,t')|^2)^\frac{1}{2}(\int_{M_s^t}\eta^4|\Delta\partial_tu|^2+\eta^4|\partial_tu|^2
+\frac{1}{R^4}\int_{M_s^t}|\partial_tu|^2)^\frac{1}{2}
\\&\quad \times(C+\frac{C\delta}{R^4})^\frac{1}{2}\nonumber
\\&\leq C\epsilon_1^\frac{1}{2}(\int_M\eta^4|\partial_tu(\cdot,t)|^2+\int_{M_s^t}\eta^4|\Delta\partial_tu|^2+C).
\end{align}
Similarly,
\begin{eqnarray*}
I_2=\int_{M_s^t}\eta^4|\nabla^2u|^2|\partial_tu|^2
\leq \int_s^t(\int_M\eta^8|\partial_tu|^4)^\frac{1}{2}(\int_{B^M_{2R}(x_0)}|\nabla^2u|^4)^\frac{1}{2}.
\end{eqnarray*}
By \eqref{ine6}, we get
\begin{eqnarray*}
(\int_{B^M_{2R}(x_0)}|\nabla^2u|^4)^\frac{1}{2}
&\leq& C(\int_{B^M_{2R}(x_0)}|\nabla^2u|^2)^\frac{1}{2}(\int_{B^M_{2R}(x_0)}|\nabla^4u|^2+\frac{1}{R^4})^\frac{1}{2}
\\&\leq& C\epsilon_1^\frac{1}{2}(\int_{B^M_{2R}(X_0)}|\nabla^4u|^2+\frac{1}{R^4})^\frac{1}{2}.
\end{eqnarray*}
Then, by the same argument as in the estimates of $I_1$, we get
\begin{eqnarray}\label{ine13}
I_2\leq C\epsilon_1^\frac{1}{2}(\int_M\eta^4|\partial_tu(\cdot,t)|^2+\int_{M_s^t}\eta^4|\Delta\partial_tu|^2+C).
\end{eqnarray}
For $I_3$, we have
\begin{eqnarray*}
I_3=\int_s^t\int_M\eta^4|\partial_tu|^2|\nabla u|^4
\leq \int_s^t(\int_M\eta^8|\partial_tu|^4)^\frac{1}{2}(\int_{B^M_{2R}(x_0)}|\nabla u|^8)^\frac{1}{2}.
\end{eqnarray*}
By \eqref{ine7}, we get
\begin{equation*}
(\int_{B_{B^M_{2R}(X_0)}}|\nabla u|^8)^\frac{1}{2}\leq C(\int_{B^M_{2R}}|\nabla u|^4)^\frac{1}{2}(\int_{B^M_{2R}}|\nabla^4u|^2+\frac{1}{R^4})^\frac{1}{2}.
\end{equation*}
Then, by the same arguments as in the estimates of $I_1, I_2$,  we obtain
\begin{eqnarray}\label{ine14}
I_3\leq C\epsilon_1^\frac{1}{2}(\int_M\eta^4|\partial_tu(\cdot,t)|^2+\int_{M_s^t}\eta^4|\Delta\partial_tu|^2+C).
\end{eqnarray}
Combining inequalities \eqref{ine11}-\eqref{ine14} yields
\begin{eqnarray}\label{ine15}
&&\frac{1}{2}\int_{M_s^t}\eta^4\partial_t|\partial_tu|^2+\int_{M_s^t}\eta^4|\Delta\partial_tu|^2
+2\int_{M_s^t}\nabla\eta^4\nabla\partial_tu\Delta\partial_tu+\int_{M_s^t}\Delta\eta^4\partial_tu\Delta\partial_tu\nonumber
\\&&\leq C\epsilon_1^\frac{1}{2}(\int_M\eta^4|\partial_tu(\cdot,t)|^2+\int_{M_s^t}\eta^4|\Delta\partial_tu|^2+C).
\end{eqnarray}
By the Cauchy-Schwartz inequality, we have
$$2|\int_{M_s^t}\nabla\eta^4\nabla\partial_tu\Delta\partial_tu|\leq\frac{1}{4}\int_{M_s^t}\eta^4|\Delta\partial_tu|^2
+64\int_{M_s^t}\eta^2|\nabla\eta|^2|\nabla\partial_tu|^2$$ and
\begin{align}\label{ine17}
\int_{M_s^t}\Delta\eta^4\partial_tu\Delta\partial_tu&=
\int_{M_s^t}(4\eta^3\Delta\eta+12\eta^2|\nabla\eta|^2)\partial_tu\Delta\partial_tu\notag\\
&\geq-\frac{1}{4}\int_{M_s^t}\eta^4|\Delta\partial_tu|^2-\frac{C}{R^4}\int_{M^t}|\partial_tu|^2\notag\\
&\geq-\frac{1}{4}\int_{M_s^t}\eta^4|\Delta\partial_tu|^2-\frac{C}{R^4}.
\end{align}
Furthermore, by integration by parts and noting that $\partial_tu|_{\partial M}=0$, we have
\begin{align*}
64\int_{M_s^t}\eta^2|\nabla\eta|^2|\nabla\partial_tu|^2
&=-64\int_{M_s^t}\nabla(\eta^2|\nabla\eta|^2)\nabla\partial_tu\partial_tu
-64\int_{M_s^t}\eta^2|\nabla\eta|^2\Delta\partial_tu\partial_tu
\\&\leq \int_{M_s^t}\eta^2|\nabla\eta|^2|\nabla\partial_tu|^2+\frac{1}{8}\int_{M_s^t}\eta^4|\Delta\partial_tu|^2+\frac{C}{R^4}\int_{M_s^t}|\partial_tu|^2.
\end{align*}
Therefore we get
\begin{eqnarray}\label{ine16}
2\int_{M_s^t}\nabla\eta^4\nabla\partial_tu\Delta\partial_tu\geq-\frac{1}{2}\int_{M_s^t}\eta^4|\Delta\partial_tu|^2-\frac{C}{R^4}.
\end{eqnarray}
Combining inequalities \eqref{ine15}, \eqref{ine16} and \eqref{ine17} and choosing $\epsilon_1$
sufficiently small,  we can finally achieve \eqref{inemain}. \endproof

Now we can derive an $L^2$ estimate for $\nabla^4u$ by the above lemma.

\begin{lem}\label{Lemma:W^{4,2}}
Let $u\in V(M^T)\cap_{\sigma>0}C^4(M_\sigma^T;N)$ be a solution of \eqref{equ1}-\eqref{equ2}. Assume that there exists $R>0$ such that $$\sup_{0\leq t<T}E(u(t);B^M_{4R}(x_0))\leq\epsilon_1.$$ Then there exists $0<\delta<min\{T, CR^4\}$ such that for all $t'\in [\frac{3\delta}{4}, T)$  we have
\begin{eqnarray}
\int_{B^M_{\frac{R}{2}}(x_0)}|\nabla^4u(\cdot,t')|^2\leq C(\frac{1}{\delta}+\frac{1}{R^4}).
\end{eqnarray}
\end{lem}
\proof By Lemma \ref{lem:02}, we have, for all $t'\in [\frac{3\delta}{4}, T)$
\begin{eqnarray}
\int_{B^M_{\frac{R}{2}}(x_0)}|\nabla^4u|^2(\cdot,t')dx\leq C\int_{B^M_{R}(x_0)}|\partial_tu|^2(\cdot,t')dx+\frac{C}{R^4}.
\end{eqnarray}
Let $\eta$ be a cut off function as in Lemma \ref{lem:03}. Without loss of generality, we assume that
$$\int_M\eta^4|\partial_tu(\cdot,s)|^2dx=\inf_{t'-\frac{\delta}{2}\leq s'\leq t'-\frac{\delta}{4}}\int_M\eta^4|\partial_tu(\cdot,s')|^2.$$
Then Lemma \ref{lem:03} implies
\begin{align*}
\sup_{t'-\frac{\delta}{4}\leq t\leq t'}\int_M\eta^4|\partial_tu(\cdot,t)|^2
&\leq C\int_M\eta^4|\partial_tu(\cdot,s)|^2+\frac{C}{R^4}
\\&= C\inf_{t'-\frac{\delta}{2}\leq s'\leq t'-\frac{\delta}{4}}\int_M\eta^4|\partial_tu(\cdot,s')|^2+\frac{C}{R^4}
\\&\leq \frac{C}{\delta}\int_{t'-\frac{\delta}{2}}^{t'-\frac{\delta}{4}}\int_M|\partial_tu|^2+\frac{C}{R^4}
\leq C(\frac{1}{\delta}+\frac{1}{R^4}).
\end{align*}
Therefore we have
\begin{eqnarray}
\int_{B^M_{R}(x_0)}|\partial_tu|^2(\cdot,t')dx\leq C(\frac{1}{\delta}+\frac{1}{R^4}).
\end{eqnarray}
This completes the proof.
\endproof

By Lemma \ref{Lemma:W^{4,2}}, we have

\begin{thm}\label{thm:C^k}
There exists $\epsilon_1>0$, depending only on $M,N,u_0,g$, such that for $0<T<\infty$, if $u$ is a smooth solution of \eqref{equ1}-\eqref{equ2} satisfying
\begin{eqnarray}
\sup_{0<t\leq T}\int_{B^M_{4R}(x_0)}|\nabla^2u(\cdot,t)|^2dx+(\int_{B^M_{4R}(x_0)}|\nabla u(\cdot,t)|^4dx)^\frac{1}{2}\leq\epsilon_1,
\end{eqnarray}
for some $R<\frac{1}{2}inj_M$ and $x_0\in M$, then we have
\begin{eqnarray}
\max_{\frac{T}{2}\leq t\leq T}\|u\|_{C^k(B^M_{\frac{R}{4}}(x_0))}\leq C\Big(k,R^{-1},T,\|\nabla^2u_0\|_{L^2(M)},\|g\|_{C^{k}(\partial M)}, \|h\|_{C^{k-1}(\partial M)}\Big).
\end{eqnarray}
\end{thm}
\proof It follows from Lemma \ref{Lemma:W^{4,2}} that $u$ is uniformly bounded in $W^{4,2}(B^M_{\frac{R}{2}}(x_0))$ for $\frac{T}{2}\leq t\leq T$. It follows that $u_t+\Delta^2u\in L^p(B^M_{\frac{R}{2}}(x_0)\times[\frac{T}{2},T])$ for any $1<p<\infty$. Therefore by the standard parabolic $L^p$-theory and Schauder estimate,
we can get the desired estimate.
\endproof

To prove our main theorem, we need to establish a lower bound estimate of the time interval for the existence of a smooth solution of \eqref{equ1}-\eqref{equ2}. First we have

\begin{lem}\label{lem:04}
Let $u\in V(M^T)$ be a solution of \eqref{equ1}-\eqref{equ2}. Assume that there exists $0<R<1$ such that $$\sup_{0\leq t<T}E(u(t);B^M_{2R}(x_0))\leq\epsilon_1.$$ Then we have for all $t\in [0,T)$
\begin{equation}\label{equation:01}
E(u(t);B^M_R(x_0))\leq CE(u(0);B^M_{2R}(x_0))+\frac{Ct}{R^4}+\frac{C\sqrt{t}}{ R^2}+CR^2.
\end{equation}
\end{lem}
\begin{proof}
Multiplying \eqref{Heat-Flow} by $\eta^4\partial_tu$ and integrating by parts, we get
\begin{align*}
&\int_{M^t}\eta^4|\partial_tu|^2\,dxdt+\frac{1}{2}\int_{M^t}\eta^4\frac{\partial}{\partial t}|\Delta u|^2\,dxdt\\
&=
\int_{M^t}\Delta u \Delta \eta^4\partial_tu\,dxdt+2\int_{M^t}\Delta u\nabla\eta^4\nabla\partial_tu\,dxdt\\
&=
-\int_{M^t}\Delta u \Delta \eta^4\partial_tu\,dxdt-2\int_{M^t}\nabla\Delta u\nabla\eta^4\partial_tu\,dxdt\\
&\leq
\frac{1}{2}\int_{M^t}\eta^4|\partial_tu|^2\,dxdt+\frac{C}{R^4}\int_{{B^M_{2R}(x_0)}^t}|\Delta u|^2 \,dxdt+\frac{C}{R^2}\int_{{B^M_{2R}(x_0)}^t}|\nabla^3 u|^2 \,dxdt\\
&\leq
\frac{1}{2}\int_{M^t}\eta^4|\partial_tu|^2\,dxdt+\frac{C}{R^4}\int_{{B^M_{2R}(x_0)}^t}|\Delta u|^2 \,dxdt+C\epsilon\int_{{B^M_{2R}(x_0)}^t}|\nabla^4 u|^2 \,dxdt\\&\quad+\frac{C}{\epsilon R^4}\int_{{B^M_{2R}(x_0)}^t}|\nabla^2 u|^2 \,dxdt\\
&\leq
\frac{1}{2}\int_{M^t}\eta^4|\partial_tu|^2\,dxdt+\frac{Ct}{R^4}+\frac{C\sqrt{t}}{ R^2}\quad \mbox{ by taking }\epsilon=\frac{\sqrt{t}}{R^2}.
\end{align*}
Then we have
\begin{align}
\int_{M^t}\eta^4|\partial_tu|^2\,dxdt+\int_{M^t}\eta^4\frac{\partial}{\partial t}|\Delta u|^2\,dxdt\leq
\frac{Ct}{R^4}+\frac{C\sqrt{t}}{ R^2}.
\end{align}
Thus we obtain
\begin{align}
\int_{B^M_R(x_0)}|\Delta u|^2(t)\,dx\leq \int_{B^M_{2R}(x_0)}|\Delta u|^2(0)\,dx+\frac{Ct}{R^4}+\frac{C\sqrt{t}}{ R^2}.
\end{align}
Observe that
\begin{align*}
\partial_t(\frac{1}{2}|\nabla u|^2)&=\langle\nabla u, \nabla\partial_tu\rangle=\nabla\langle\nabla u, \partial_tu\rangle-\langle\Delta u, \partial_tu\rangle.
\end{align*}
Multiplying this equality by $\eta^4$, integrating it over $M$,
and applying integration by parts and  \eqref{Heat-Flow}, we obtain
\begin{align*}
\int_{M^t}\partial_t(\frac{1}{2}\eta^4|\nabla u|^2)dxdt&=\int_{M^t}\eta^4\nabla\langle\nabla u, \partial_tu\rangle dxdt- \int_{M^t}\eta^4\langle\Delta u, \partial_tu\rangle dxdt\\
&=-\int_{M^t}\nabla\eta^4\langle\nabla u, \partial_tu\rangle dxdt- \int_{M^t}\eta^4\langle\Delta u, \partial_tu\rangle dxdt\\
&\leq
\frac{R^2}{4}\int_{M^t}\eta^4|\partial_tu|^2dxdt+\frac{C}{R^4}\int_{M^t}\eta^2|\nabla u|^2dxdt+\frac{C}{R^2}\int_{M^t}\eta^4|\Delta u|^2dxdt\\
&\leq
\frac{R^2}{4}\int_{M^t}\eta^4|\partial_tu|^2dxdt+\frac{C}{R^2}\int_0^{t}(\int_{B_{2R}(x_0)}|\nabla u|^4dx)^{\frac{1}{2}}dt+\frac{Ct}{R^2}\\
&\leq
\frac{R^2}{4}\int_{M^t}\eta^4|\partial_tu|^2dxdt+\frac{Ct}{R^2}\\
&\leq \frac{R^2}{4}\int_{B^M_{2R}(x_0)}|\Delta u|^2(0)dx+C\sqrt{t}+\frac{Ct}{R^2}.
\end{align*}
Thus,
\begin{align}
\int_{B^M_R(x_0)}|\nabla u|^2(t)dx\leq \int_{B^M_{2R}(x_0)}|\nabla u|^2(0)dx+\frac{R^2}{4}\int_{B^M_{2R}(x_0)}|\Delta u|^2(0)dx+C\sqrt{t}+\frac{Ct}{R^2}.
\end{align}
Let $q\in C^{2+\alpha}(M,\mathbb{R}^N)$ be a harmonic function, satisfying
\begin{align*}
\begin{cases}
\Delta q=0 \ &\ {\rm{in}} \ M,\\
q=g\ &\ {\rm{on}}\ \partial M.
\end{cases}
\end{align*}
Then we have
\[
\|q\|_{C^{2+\alpha}(M)}\leq C(M)\|g\|_{C^{2+\alpha}(M)},
\]
and hence
\begin{align*}
&\int_{B^M_R(x_0)}|\nabla^2 (u-q)|^2(t)dx+(\int_{B^M_R(x_0)}|\nabla (u-q)|^4(t)dx)^{\frac{1}{2}}\\
&\leq
C\int_{B^M_{\frac{3}{2}R}(x_0)}|\Delta (u-q)|^2(t)dx+\frac{C}{R^2}\int_{B^M_{\frac{3}{2}R}(x_0)}|\nabla (u-q)|^2(t)dx\\
&\leq
C\int_{B^M_{2R}(x_0)}|\Delta u|^2(0)dx+\frac{C}{R^2}\int_{B^M_{2R}(x_0)}|\nabla u|^2(0)dx+\frac{Ct}{R^4}+\frac{C\sqrt{t}}{ R^2}+\frac{C}{R^2}\int_{B^M_{2R}(x_0)}|\nabla q|^2dx\\
&\leq
C\int_{B^M_{2R}(x_0)}|\Delta u|^2(0)dx+C(\int_{B^M_{2R}(x_0)}|\nabla u|^4(0)dx)^{\frac{1}{2}}+\frac{Ct}{R^4}+\frac{C\sqrt{t}}{ R^2}+CR^2.
\end{align*}
This implies
\begin{align*}
&\int_{B^M_R(x_0)}|\nabla^2 u|^2(t)dx+(\int_{B^M_R(x_0)}|\nabla u|^4(t)dx)^{\frac{1}{2}}\\
&\leq
C(\int_{B^M_R(x_0)}|\nabla^2 (u-q)|^2(t)dx+\int_{B^M_R(x_0)}|\nabla^2 q|^2(t)dx)\\&\quad+C((\int_{B^M_R(x_0)}|\nabla (u-q)|^4(t)dx)^{\frac{1}{2}}+(\int_{B^M_R(x_0)}|\nabla q|^4dx)^{\frac{1}{2}})\\
&\leq
C\int_{B^M_{2R}(x_0)}|\Delta u|^2(0)dx+C(\int_{B^M_{2R}(x_0)}|\nabla u|^4(0)dx)^{\frac{1}{2}}+\frac{Ct}{R^4}+\frac{C\sqrt{t}}{ R^2}+CR^2,
\end{align*}
which implies \eqref{equation:01}. This completes the proof. \end{proof}

According to Lemma \ref{lem:04}, we have

\begin{lem}\label{lem:LB}
There exists $0<\epsilon_2\ll\epsilon_1<\frac{inj_M}{2}$ such that if $u_0\in C^\infty(M,N)$,
$g\in C^\infty(\partial M, N)$, and $h\in C^\infty(\partial M, T_gN)$ satisfies
\begin{eqnarray}
\sup_{x\in M}\left(\int_{B^M_{2R}(x)}|\nabla^2u_0|^2\,dx+(\int_{B^M_{2R}(x)}|\nabla u_0|^4\,dx)^{\frac{1}{2}}\right)\leq\epsilon_2^2,
\end{eqnarray}
for some $R\in (0,\epsilon_2)$. Then there exists $T_1\geq O(R^4\epsilon_1^2)$ and a unique
solution $u\in C^\infty(M\times[0,T_1],N)$ to \eqref{equ1}-\eqref{equ2}.
\end{lem}
\begin{proof}
Let $T_1>0$ be the maximum time interval such that there exists a smooth solution $u\in C^{\infty}(M\times  [0,T_1),N)$ of \eqref{equ1}-\eqref{equ2}. Let $T'_1>0$ be the maximum time such that
\begin{equation}
\sup_{0\leq t\leq T'_1}\sup_{x\in M}E(u(t);B^M_{R}(x))\leq \epsilon_1.
\end{equation}
By Theorem \ref{thm:C^k}, we know $T_1\geq T'_1$.
By Lemma \ref{lem:04}, we get
\begin{align*}
\epsilon_1=E(u(T_1');B^M_{R}(x))&\leq CE(u(0);B^M_{2R}(x))+\frac{CT_1'}{R^4}+\frac{C\sqrt{T_1'}}{ R^2}+CR^2\\
&\leq \frac{\epsilon_1}{2}+\frac{CT_1'}{R^4}+\frac{C\sqrt{T_1'}}{ R^2}\\
&\leq \frac{3\epsilon_1}{4}+\frac{CT_1'}{\epsilon_1R^4}.
\end{align*}
This implies $T_1\geq T'_1\geq O(R^4\epsilon_1^2)$.
\end{proof}

\section{Existence results and behavior of solutions near singularities}

In this section, we show the existence of the global weak solution of the extrinsic biharmonic map flow, which is regular with the exception of at most finitely many time slices. We also study the behavior of the solution
near its singularities. Moreover, we get the existence of biharmonic maps with a fixed Dirichlet boundary data. Both Theorem \ref{main thm} and Theorem \ref{thm:main-02} will be proved in this section.

\begin{proof}[\textbf{Proof of Theorem \ref{main thm}}]

\smallskip
\noindent\textbf{Step 1}. From \cite{SU} and \cite{BVP} we see that there exists a sequence of maps
$\phi_{l}\in C^{4+\alpha}(M, N)$ such that $\phi_l=g, \partial_\nu\phi_l=h$ on $\partial M$, and
\begin{align*}
\phi_{l}\to u_0 \quad &{\rm{strongly}} \quad {\rm{in}} \quad W^{2,2}(M,N).
\end{align*}

\noindent\textbf{Step 2}. The short-time existence. Since $\phi_l\to u_0$ in $W^{2,2}(M,N)$,
there exists a $R\in (0,\frac{inj_M}{2})$ such that
$$\sup_l\sup_{x\in M}\left(\int_{B_{2R}(x)}|\nabla^2\phi_l|^2\,dx+(\int_{B_{2R}(x)}|\nabla \phi_l|^4\,dx)^{\frac{1}{2}}\right)\leq\epsilon_2^2,$$
where $\epsilon_2$ is given in Lemma \ref{lem:LB}. By the short-time existence theory in \cite{Lamm}, there exist $T_l>0$ and $u_l\in C^{4+\alpha,1+\frac{\alpha}{4}}(M\times[0,T_l),N)$ which solves  \eqref{Heat-Flow} with the boundary-initial data $(g,h)$. Then Lemma \ref{lem:LB} implies that $T_l\geq O(R^4\epsilon_1^2)$ and Theorem \ref{thm:C^k} implies that we have uniformly $C_{loc}^{4+\alpha,1+\frac{\alpha}{4}}$ estimates of $u_l$ in $M\times(0,O(R^4\epsilon_1^2)]$. Hence we may assume that $u_l$ converges to $u$ weakly in $W^{2,2}(M,N)$, strongly in $W^{1,2}(M,N)$ and in $C^{4+\alpha,1+\frac{\alpha}{4}}(M\times[\rho,O(R^4\epsilon_1^2)],N)$ for any $\rho>0$. It is clear that $u\in C_{loc}^{4+\alpha,1+\frac{\alpha}{4}}(M\times(0,O(R^4\epsilon_1^2)),N)$ is a classical solution of \eqref{Heat-Flow}. The short-time existence theory guarantees the existence of a solution to \eqref{Heat-Flow} using $u(O(R^4\epsilon_1^2))$ as the new initial data so that the solution can be continued to a larger time interval. Assume that $T_1$ is the maximum time interval such that $u\in C_{loc}^{4+\alpha,1+\frac{\alpha}{4}}(M\times(0,T_1),N)$ solves \eqref{equ1}-\eqref{equ2}. Repeating this argument, the solution can be continued  until the first time of energy concentration excels $\epsilon_1$, that is, the condition
\begin{equation}\label{equation:04}
\lim_{r\downarrow0}\lim_{t\uparrow T_1}\sup_{x\in M}E(u(t),B_{r}(x))>\epsilon_1
\end{equation}
reaches. Set
\begin{equation}\label{equation:05}
S(T_1):=\Big\{(x,T_1)\ |\ x\in M,\ \lim_{r\to 0}\limsup_{t\uparrow T_1}E(u(t),B_{r}(x))>\epsilon_1\Big\},
\end{equation}
which is called as the singularity set of $u$ at time $T_1$. It is an open question if $S(T_1)$ is
a finite set.

\noindent{\textbf{Step 3}}. Behavior of the solution $u$ near its first singular time $T_1$.
By the standard blowup argument, there exist sequences $t_i^1\nearrow T_1$, $x_i^1\to x_0\in\overline M$, and
$r_i^1\to 0$ such that
\begin{align}\label{equation:blowup-position}
E(u(t_i^1),B_{r_i^1}^M(x_i^1))=\sup_{\substack{(x,t)\in \overline M\times [T_1-\delta^2,t_i^1]\\
r\leq r_i^1}}E(u(t),B_{r}^M(x))=\frac{\epsilon_1}{C_0},
\end{align}
where $C_0$ is a positve constant to be determined later. Assume that $B_{2r_i^1}(x_0)$ is covered by $m$ balls of radius $r_i^1$ constained in $\overline M$ and let $C_0=m$, then we see that $\sup_{T_1-\delta^2\leq t<T_1}E(u(t);B_{2r^1_i}^M(x_0))\leq \epsilon_1$.
By Lemma \ref{lem:04}, for any $T_1-\delta^2\leq s\leq t_i^1<T_1$, we have
\begin{align*}
E(u(t_i^1);B_{r_i^1}^M(x_0))\leq CE(u(s);B_{2r^1_i}^M(x_0))+C\frac{t_i^1-s}{(r_i^1)^4}+C\frac{\sqrt{t_i^1-s}}{(r_i^1)^2}+C(r_i^1)^2.
\end{align*}
Set $T=\frac{\epsilon_1^2}{16C^2C_0^2}$. Then we have
\begin{equation}\label{equ3}
E(u(s);B_{2r_i^1}^M(x_0))\geq\frac{\epsilon_1}{2CC_0}
\end{equation}
for any $s\in [t_i^1-T(r_i^1)^4,t_i^1]$, when $i$ is sufficiently large.

\noindent{\textbf{Case 1}}. $\displaystyle\limsup_{i\to\infty}\frac{{\rm{dist}}(x_i^1,\partial M)}{r_i^1}\to \infty.$
By passing to a subsequence, we may assume
$\lim_{i\to\infty}\frac{{\rm{dist}}(x_i^1,\partial M)}{r_i^1}\to \infty$.
Assume $t_i^1-\frac{\delta^2}{4}>T_1-\delta^2$ and define
\[
B_i:=\big\{x\in\mathbb{R}^4|x_i^1+r_i^1x\in B_\delta^M(x_0)\big\},
\]
and
\begin{align*}
v_i(x,t):&=u(x_i^1+r_i^1x,t_i^1+(r_i^1)^4t), \ \forall\ x\in B_i, \ -\frac{\delta^2}{4(r_i^1)^4}\le t\le 0.
\end{align*}
It is easy to see that $B_i\rightarrow\R^4$ as $i\rightarrow\infty$.
Then $v_i$ satisfies
\begin{eqnarray}\label{equation:blowup}
\partial_tv_i+\Delta^2v_i=-f(v_i),
\end{eqnarray}
along with the boundary condition
\begin{eqnarray}
\begin{cases}
v_i(x,t)=g(x_i^1+r_i^1x),\quad &if\quad x_i^1+r_i^1x\in \partial M;\\
\partial_\nu v_i(x,t)=r_i^1h(x_i^1+r_i^1 x),\quad &if\quad x_i^1+r_i^1x\in \partial M.
\end{cases}
\end{eqnarray}
By Lemma \ref{lem:01}, we have
\begin{align}\label{inequality:01}
\int_{-T}^0\int_{B_i}|\partial_tv_i|^2\,dxdt\leq\int_{t_i^1-(r_i^1)^4T}^{t_i^1}\int_M|\partial_tu|^2\,dv_gdt
\to 0 \quad {\text as}\ \ i\to \infty,
\end{align}
and
\begin{equation}\label{inequality:02}
\sup_{\frac{\delta^2}{4(r_i^1)^4}\leq t\leq0}E(v_i,B_i)\leq\sup_{T_1-\delta^2\leq t\leq T_1}E(u)\leq C.
\end{equation}
By \eqref{equation:blowup-position}, we can see that
\begin{align*}
\sup_{-T\leq t\leq 0}\sup_{x\in B_i}E(v_i,B_1(x)\cap B_i)\leq
\sup_{\substack{(x,t)\in \overline M\times [T_1-\delta^2,t_i^1]\\
\ r\leq r_i^1}}E(u(t),B_{r}(x))=\frac{\epsilon_1}{2C_0}.
\end{align*}
Hence, for any $x\in \mathbb{R}^4$, when $i$ is sufficiently large, we have
\begin{equation}\label{equ4}
\sup_{-T\leq t\leq 0}E(v_i,B_1(x))\leq\frac{\epsilon_1}{2C_0}.
\end{equation}
Combining \eqref{equ4} with Theorem \ref{thm:C^k}, we have
\begin{equation}\label{equ5}
\sup_{-\frac{T}{2}\leq t\leq 0}\|v_i(\cdot,t)\|_{C^{4+\alpha}(B_{1/2}(x))}\leq C,
\end{equation}
which yields
\begin{equation}\label{equ6}
\sup_{-\frac{T}{2}\leq t\leq 0}\|v_i(\cdot,t)\|_{C^{4+\alpha}(B_R)}\leq C(R), \ \forall\ R>0.
\end{equation}
From \eqref{inequality:01} and \eqref{equ6}, we can find $\sigma_i\in[-\frac{T}{2},0]$ such that as $i\to\infty$, there holds
\begin{align}
\int_{B_i}|\partial_tv_i|^2(x,\sigma_i)\,dx \to 0
\end{align}
and
\begin{equation}
\|v_i(\cdot,\sigma_i)\|_{C_{loc}^{4+\alpha}(\mathbb{R}^4)}\leq C.
\end{equation}
Therefore, there exists a subsequence of $v_i(\cdot,\sigma_i)$ and a limit map
$v\in C^4(\mathbb{R}^4, N)$ such that
\begin{align}\label{equation:02}
v_i(\cdot,\sigma_i)&\to v\quad {\rm{in}} \quad C^4(B_R), \ \forall \ R>0.
\end{align}
Setting $t=\sigma_i$ in the equation \eqref{equation:blowup} and letting $i\to\infty$, it is easy to see that $v$ is a biharmonic map with
\[
\frac{\epsilon_1}{2CC_0}\leq E(v;\mathbb{R}^4)\leq C,
\]
where the above inequality follows from \eqref{equ3} and \eqref{inequality:02}. Taking $t_i^1+(r_i^1)^4\sigma_i$ as the new time sequence, then we get that
\begin{align*}
u_i(x)&=v_i(x,\sigma_i)=u(x_i^1+r_i^1x,t_i^1+(r_i^1)^4\sigma_i)
\end{align*}
is the desired sequence in the theorem.

\noindent\textbf{Case 2}. $\displaystyle\limsup_{i\to\infty}\frac{dist(x_i^1,\partial M)}{r_i^1}< \infty$.
After taking a subsequence, we may assume $\frac{dist(x_i^1,\partial M)}{r_i^1}\to a$ as $i\to\infty$. Then
\[
B_i\to \mathbb{R}^4_a:=\big\{(x',x^4)|x^4\geq -a\big\},
\]
where $x':=(x^1,x^2,x^3)\in\mathbb{R}^3$. Noting that for any $x\in\big\{x^4=-a\big\}$, $x_i^1+r_i^1x\to x_0$.
Moreover,
\begin{eqnarray}
\begin{cases}
v_i(x,t)=g(x_i^1+r_i^1x),\quad &if\quad x_i^1+r_i^1x\in \partial M;\\
\partial_\nu v_i(x,t)=r_i^1h(x_i^1+r_i^1x),\quad &if\quad x_i^1+r_i^1x\in \partial M.
\end{cases}
\end{eqnarray}
By Theorem \ref{thm:C^k} and \eqref{equation:blowup-position}, for any $B_R(0)\subset \mathbb{R}^4,R>0$, we have
\begin{equation}
\sup_{-\frac{T}{2}\leq t\leq 0}\|v_i(\cdot,t)\|_{C^{4+\alpha}(B_R(0)\cap B_i)}
\leq C.
\end{equation}
Using a similar argument as in \textbf{Case 1}, we can obtain $v\in C^4(\mathbb{R}^4_a, N)$ satisfying
\begin{align}
\frac{\epsilon_1}{2CC_0}\leq E(v;\mathbb{R}^4_a)\leq C,
\end{align}
and a sequence $\sigma_i\in[-\frac{T}{2},0]$ such that as $i\to\infty$, there hold
\begin{align}\label{equation:03}
\|v_i(\cdot,\sigma_i)-v\|_{C^4(B_i\cap B_R(0))}&\to 0,
\end{align}
for any $R>0$.
Moreover, $v$ is a biharmonic map with the boundary condition
\begin{eqnarray}
\begin{cases}
v(x)=g(x_0),\quad &{\rm{on}}\quad \partial \mathbb{R}^4_a;\\
\partial_\nu v(x) =0,\quad &{\rm{on}}\quad \partial \mathbb{R}^4_a.
\end{cases}
\end{eqnarray}

\noindent\textbf{Step 4}. Global existence of weak solutions.
Let $(x_0,T_1)\in S(T_1)$ be as in Step 3. Then we claim
\begin{eqnarray}\label{ine18}
\lim_{r\to0}\limsup_{t\uparrow T_1}\int_{B_r(x_0)}|\Delta u(\cdot,t)|^2\,dx\geq\epsilon_0^2.
\end{eqnarray}
In fact, by \eqref{equation:02}, \eqref{equation:03} and Theorem \ref{thm:Gap-theorem}, we have
\begin{align*}
\epsilon_0^2\leq \lim_{R\to\infty}\int_{B_R(0)}|\Delta v|^2dx&=\lim_{R\to\infty}\lim_{i\to\infty}\int_{B_R(0)}|\Delta v_i|^2(\cdot,\sigma_i)dx\\
&=\lim_{R\to\infty}\lim_{i\to\infty}\int_{B_{r_iR}(x_i)}|\Delta u|^2(\cdot,t_i+r_i^4\sigma_i)dx\\
&\leq
\lim_{r\to0}\limsup_{t\uparrow T_1}\int_{B_r(x_0)}|\Delta u(\cdot,t)|^2dx.
\end{align*}
Next, we claim that there is a unique weak limit $u(\cdot,T_1)\in W^{2,2}(M,N)$ such that
$$\lim_{t\uparrow T_1}u(\cdot,t)=u(\cdot,T_1)\ {\rm{weakly\ in}}\ W^{2,2}(M,N).$$
In fact, by Lemma \ref{lem:01}, for any sequence $t_i\to T_1$, there exists a subsequence (also denoted by $t_i$) such that $u(\cdot,t_i)\to u(\cdot,T_1)$ weakly in $W^{2,2}(M)$ as $i\to\infty$. So, we just need to show the weak limit $u(\cdot,T_1)$ is independent of the choice of the time sequences. Let $s_i\to T_1$ be  another time sequence and the corresponding weak limit  $\widehat{u}(\cdot,T_1)$.  Note that
\begin{align}\label{equation:06}
&\int_M|u(\cdot,T_1)-\widehat{u}(\cdot,T_1)|^2\,dx\notag\\
&=\int_M\langle u(\cdot,T_1)-\widehat{u}(\cdot,T_1),
u(\cdot,T_1)-u(\cdot,t_i)\rangle \,dx+\int_M\langle u(\cdot,T_1)-\widehat{u}(\cdot,T_1),
u(\cdot,t_i)\notag\\&\quad-u(\cdot,s_i)\rangle \,dx+\int_M\langle u(\cdot,T_1)-\widehat{u}(\cdot,T_1),
u(\cdot,s_i)-\widehat{u}(\cdot,T_1)\rangle \,dx.
\end{align}
Since
\begin{align*}
\int_M|u(\cdot,t_i)-u(\cdot,s_i)|^2\,dx
=\int_M|\int_{s_i}^{t_i}\frac{\partial u}{\partial t}\,dt|^2\,dx
\leq |s_i-t_i||\int_{M_{s_i}^{t_i}}|\frac{\partial u}{\partial t}|^2\,dxdt|,
\end{align*}
$\displaystyle\int_{M^{T_1}}|\frac{\partial u}{\partial t}|^2\,dxdt\leq C$ (see Lemma \ref{lem:01}),
$u(\cdot,t_i)\rightharpoonup u(\cdot,T_1),\ u(\cdot,s_i)\rightharpoonup \widehat {u}(\cdot,T_1)$
weakly in $W^{2,2}(M)$
by sending $i\to\infty$ in \eqref{equation:06},
we obtain $$\int_M|u(\cdot,T_1)-\widehat{u}(\cdot,T_1)|^2\,dx=0.$$
Thus $u(\cdot,T_1)=\widehat{u}(\cdot,T_1)$.  It is easy to see
that
$$\int_M|\Delta u(\cdot, T_1)|^2\,dx
\le \int_M |\Delta u_0|^2\,dx-\epsilon_0^2.
$$
Now we use $u(\cdot,T_1)$ as the initial condition and ($g,h$) as the boundary condition
to extend the above solution beyond $T_1$ to obtain a weak solution $u:M\times(0,T_2)\to N$ for some $T_2>T_1$ by piecing together the solutions at $T_1$. Then we see that $u\in C_{loc}^{4+\alpha,1+\frac{\alpha}{4}}(M\times((0,T_2)\setminus\{T_1\}),N)$. Iterating this process, we obtain a global solution defined on $M\times [0,\infty)$. Let $\{T_k\}_{k=1}^K$ be all the possible singular times.
Then we have
\begin{align*}
\int_{M}|\Delta u(\cdot,T_K)|^2\,dx
&\leq\liminf_{t_i\uparrow T_K}\int_M|\Delta u(\cdot,t_i)|^2\,dx-\epsilon_0^2\\
&\leq
\int_M|\Delta u_0|^2\,dx-K\epsilon_0^2,
\end{align*}
which implies
\[
K\leq \frac{\int_M|\Delta u_0|^2\,dx}{\epsilon_0^2}.
\]
Hence there are at most finitely many singular time slices.

\noindent\textbf{Step 5}. Uniqueness.
The only thing left to be proven is uniqueness and we only need to prove uniqueness of the short time solution constructed above and the full uniqueness follows by iteration. Let $u,v:[0,t_0)\to N$ be two constructed
smooth (for $t>0$) solutions and set $w:=u-v$. Then
\begin{eqnarray*}
\partial_tw+\Delta^2w&=&f(u)-f(v)
\\&=&\nabla^3u\#\nabla u+\nabla^2 u\#\nabla^2 u+\nabla^2 u\#\nabla u\#\nabla u+\nabla u\#\nabla u\#\nabla u\#\nabla u
\\&&-\big(\nabla^3v\#\nabla v+\nabla^2 v\#\nabla^2 v+\nabla^2 v\#\nabla v\#\nabla v+\nabla v\#\nabla v\#\nabla v\#\nabla v\big).
\end{eqnarray*}
Multiply this equation with $w$ and integrate over $(0,s)\times  M$. By partial integration ($w=\partial_\nu w=0$ on $\partial M$ for any $t\in (0,s)$), we can get rid of derivatives of order $>2$ (cf. \cite{Gastel}). Simplifying terms by using Young's inequality we get
\begin{eqnarray}\label{ine19}
&&\frac{1}{2}\int_M|w(s)|^2+\int_{M^s}|\Delta w|^2\nonumber
\\&\leq&C\sum_{k=1}^2\int_{M^s}|w|^2(|\nabla^ku|+|\nabla^kv|)^\frac{4}{k}+C\sum_{k=1}^2\int_{M^s}|\nabla w|^2(|\nabla^ku|+|\nabla^kv|)^\frac{2}{k}\nonumber
\\&&+C\sum_{l=0}^1\sum_{k=1}^{2-l}\int_{M^s}|\nabla^2w||\nabla^lw|(|\nabla^ku|+|\nabla^kv|)^\frac{2-l}{k}\nonumber
\\&:=&I_4+I_5+I_6.
\end{eqnarray}
To make it more clear how the above inequality is obtained, let us give the details of the estimates of the highest order term of $(f(u)-f(v))w$ (we denote it by $\varphi(u)\cdot\nabla^3u\cdot \nabla u$) as follows.
\begin{align*}
&\int_{M^s}(\varphi(u)\cdot\nabla^3u\cdot \nabla u-\varphi(v)\cdot\nabla^3v\cdot\nabla v)w
\\&=\int_{M^s}[(\varphi(u)-\varphi(v))\cdot\nabla^3u\cdot \nabla u+\varphi(v)\cdot\nabla^3w\cdot \nabla u+\varphi(v)\cdot\nabla^3v\cdot \nabla w ]w
\\&=\int_{M^s}\nabla^2u\#\nabla w\# \nabla u\# w+\nabla^2u\#\nabla^2 u\nabla\# w^2+\nabla^2w\#\nabla u\# \nabla v\# w+\nabla^2w\#\nabla^2 u\# w
\\&\quad+\nabla^2w\#\nabla u\# \nabla w+\nabla^2v\#\nabla v\# \nabla w\# w+\nabla^2v\# \nabla^2 w\# w+\nabla^2v\#\nabla w\# \nabla w
\\&\leq I_4+I_5+I_6,
\end{align*}
where $$\varphi(u)\cdot\nabla^3u\cdot \nabla u:=\varphi^{ijkl}_{AB}(u)\cdot\nabla_{ijk}^3u^A\cdot \nabla_l u^B$$ and the last inequality follows from Young's inequality and following property
\begin{align*}
|\nabla^k\varphi|\leq C(N)\quad {\rm{and}} \quad |\varphi(u)-\varphi(v)|\leq C(N)(u-v).
\end{align*}
In the following, let's estimate the right hand side of \eqref{ine19}. By H\"older's inequality, the Sobolev inequality, the Poincar\'e inequality and partial integration,  we get
\begin{align}\label{ine21}
\int_{M^s}|w|^2(|\nabla u|+|\nabla v|)^4
&\leq(\int_{M^s}|w|^4)^\frac{1}{2}(\int_{M^s}|\nabla u|^8+|\nabla v|^8)^\frac{1}{2}\nonumber
\\&\leq C(\int_{M^s}|\nabla u|^8+|\nabla v|^8)^\frac{1}{2}(\int_0^s(\int_M|\nabla w|^2+|w|^2)^2)^\frac{1}{2}\nonumber
\\&\leq C(\int_{M^s}|\nabla u|^8+|\nabla v|^8)^\frac{1}{2}(\int_0^s(\int_{M}|\nabla w|^2)^2)^\frac{1}{2}\nonumber
\\&\leq C(\int_{M^s}|\nabla u|^8+|\nabla v|^8)^\frac{1}{2}(\int_0^s\int_M|w|^2\int_M|\nabla^2w|^2)^\frac{1}{2}\nonumber
\\&\leq C(\int_{M^s}|\nabla u|^8+|\nabla v|^8)^\frac{1}{2}(\sup_{t\in (0,s)}\int_M|w(t)|^2)^\frac{1}{2}(\int_{M^s}|\nabla^2w|^2)^\frac{1}{2}\nonumber
\\&\leq C(\int_{M^s}|\nabla u|^8+|\nabla v|^8)^\frac{1}{2}(\sup_{t\in (0,s)}\int_M|w(t)|^2+\int_{M^s}|\nabla^2w|^2),
\end{align}
and similarly we have
\begin{align}\label{ine22}
&\int_{M^s}|w|^2(|\nabla^2u|+|\nabla^2v|)^2
\notag\\ &\leq C(\int_{M^s}|\nabla^2u|^4+|\nabla^2v|^4)^\frac{1}{2}(\sup_{t\in (0,s)}\int_M|w(t)|^2+\int_{M^s}|\nabla^2w|^2).
\end{align}
Now let's estimate $I_5$ term by term,
\begin{align}
\int_{M^s}|\nabla w|^2|\nabla u|^2\nonumber
&=-\int_{M^s}w\Delta w|\nabla u|^2-\int_{M^s}w\nabla w\nabla u\cdot\nabla^2u\nonumber
\\&\leq \epsilon\int_{M^s}|\nabla^2w|^2+C(\epsilon)\int_{M^s}w^2|\nabla u|^4+\frac{1}{2}\int_{M^s}|\nabla w|^2|\nabla u|^2+\int_{M^s}w^2|\nabla^2u|^2.\nonumber
\end{align}
Therefore we get
$$\int_{M^s}|\nabla w|^2|\nabla u|^2\leq\epsilon\int_{M^s}|\nabla^2w|^2+C(\epsilon)\int_{M^s}w^2|\nabla u|^4+\int_{M^s}w^2|\nabla^2u|^2.$$
Thus by \eqref{ine21} and \eqref{ine22} we have
\begin{eqnarray}
&&\int_{M^s}|\nabla w|^2(|\nabla u|^2+|\nabla v|^2)\nonumber
\\&\leq&\epsilon\int_{M^s}|\nabla^2w|^2+C(\epsilon)((\int_{M^s}|\nabla u|^8+|\nabla v|^8)^\frac{1}{2}+(\int_{M^s}|\nabla^2u|^4+|\nabla^2v|^4)^\frac{1}{2})\nonumber
\\&&\times(\sup_{t\in (0,s)}\int_M|w(t)|^2+\int_{M^s}|\nabla^2w|^2).\nonumber
\end{eqnarray}
In addition,
\begin{align*}
\int_{M^s}|\nabla w|^2|\nabla^2u|
&= -\int_{M^s}w\Delta w|\nabla^2u|-\int_{M^s}w\nabla w\cdot\nabla|\nabla^2u|
\\&\leq \epsilon\int_{M^s}|\nabla^2w|^2+C(\epsilon)\int_{M^s}w^2|\nabla^2u|^2-\frac{1}{2}\int_{M^s}\nabla w^2\cdot\nabla|\nabla^2u|
\\&=\epsilon\int_{M^s}|\nabla^2w|^2+C(\epsilon)\int_{M^s}w^2|\nabla^2u|^2+\frac{1}{2}\int_{M^s}w^2\Delta|\nabla^2u|
\\&\leq\epsilon\int_{M^s}|\nabla^2w|^2+C(\epsilon)\int_{M^s}w^2|\nabla^2u|^2+C(\int_{M^s}w^4)^\frac{1}{2}(\int_{M^s}|\nabla^4u|^2)^\frac{1}{2}
\\&\leq\epsilon\int_{M^s}|\nabla^2w|^2+C(\epsilon)(\int_{M^s}|\nabla^2u|^4)^\frac{1}{2})(\sup_{t\in (0,s)}\int_M|w(t)|^2+\int_{M^s}|\nabla^2w|^2)
\\&+C(\int_{M^s}|\nabla^4u|^2)^\frac{1}{2}(\sup_{t\in (0,s)}\int_M|w(t)|^2+\int_{M^s}|\nabla^2w|^2).
\end{align*}
Thus we get
\begin{eqnarray}
&&\int_{M^s}|\nabla w|^2(|\nabla^2u|+|\nabla^2v|)\nonumber
\\&\leq&\epsilon\int_{M^s}|\nabla^2w|^2+C(\epsilon)(\int_{M^s}|\nabla^2u|^4+|\nabla^2v|^4)^\frac{1}{2}(\sup_{t\in (0,s)}\int_M|w(t)|^2+\int_{M^s}|\nabla^2w|^2)\nonumber
\\&&+C(\int_{M^s}|\nabla^4u|^2+|\nabla^4v|^2)^\frac{1}{2}(\sup_{t\in (0,s)}\int_M|w(t)|^2+\int_{M^s}|\nabla^2w|^2).\nonumber
\end{eqnarray}
In summation of the above estimates we have
\begin{eqnarray}
I_5&\leq&\epsilon\int_{M^s}|\nabla^2w|^2+C(\epsilon)(\int_{M^s}|\nabla u|^8+|\nabla v|^8)^\frac{1}{2}(\sup_{t\in (0,s)}\int_M|w(t)|^2+\int_{M^s}|\nabla^2w|^2)\nonumber
\\&&+C(\epsilon)(\int_{M^s}|\nabla^2u|^4+|\nabla^2v|^4)^\frac{1}{2}(\sup_{t\in (0,s)}\int_M|w(t)|^2+\int_{M^s}|\nabla^2w|^2)\nonumber
\\&&+C(\int_{M^s}|\nabla^4u|^2+|\nabla^4v|^2)^\frac{1}{2}(\sup_{t\in (0,s)}\int_M|w(t)|^2+\int_{M^s}|\nabla^2w|^2).
\end{eqnarray}
We are left to estimate the last summation $I_6$ in \eqref{ine19}. Note that by Young's inequality
\begin{eqnarray}\label{ine20}
I_6\leq\epsilon\int_{M_o^s}|\nabla^2w|^2+C(\epsilon)(I_4+I_5).
\end{eqnarray}
Therefore from inequalities \eqref{ine19}-\eqref{ine20} we obtain
\begin{eqnarray}
&&\int_M|w(s)|^2+\int_{M^s}|\Delta w(s)|^2\nonumber
\\&\leq&\epsilon\int_{M^s}|\nabla^2w|^2+C(\epsilon)(\int_{M^s}|\nabla u|^8+|\nabla v|^8)^\frac{1}{2}(\sup_{t\in (0,s)}\int_M|w(t)|^2+\int_{M^s}|\nabla^2w|^2)\nonumber
\\&&+C(\epsilon)(\int_{M^s}|\nabla^2u|^4+|\nabla^2v|^4)^\frac{1}{2}(\sup_{t\in (0,s)}\int_M|w(t)|^2+\int_{M^s}|\nabla^2w|^2)\nonumber
\\&&+ C(\epsilon)(\int_{M^s}|\nabla^4u|^2+|\nabla^4v|^2)^\frac{1}{2}(\sup_{t\in (0,s)}\int_M|w(t)|^2+\int_{M^s}|\nabla^2w|^2).
\end{eqnarray}
By  the standard elliptic estimate, noting that $w=0$ on $\partial M$,
we have $\int_M|\nabla^2w|^2\leq C\int_M|\Delta w|^2$.
Choosing $\epsilon=\frac{1}{2C}$, we obtain
\begin{align}
&\int_M|w(s)|^2+\int_{M^s}|\nabla^2w(s)|^2\nonumber
\\&\leq C(\int_{M^s}|\nabla u|^8+|\nabla v|^8)^\frac{1}{2}(\sup_{t\in (0,s)}\int_M|w(t)|^2+\int_{M^s}|\nabla^2w|^2)\nonumber
\\&\quad+C(\int_{M^s}|\nabla^2u|^4+|\nabla^2v|^4)^\frac{1}{2}(\sup_{t\in (0,s)}\int_M|w(t)|^2+\int_{M^s}|\nabla^2w|^2)\nonumber
\\&\quad+ C(\int_{M^s}|\nabla^4u|^2+|\nabla^4v|^2)^\frac{1}{2}(\sup_{t\in (0,s)}\int_M|w(t)|^2+\int_{M^s}|\nabla^2w|^2).
\end{align}
 Hence for solutions $u,v\in V(M^T)$ by the interpolation inequalities of Lemma \ref{Lem:interpolation} we see that we can choose $s$ small enough such that $C(\int_{M^s}|\nabla u|^8+|\nabla v|^8)^\frac{1}{2}$, $C(\int_{M^s}|\nabla^2u|^4+|\nabla^2v|^4)^\frac{1}{4}$ and $C(\int_{M^s}|\nabla^4u|^2+|\nabla^4v|^2)^\frac{1}{2}$ are all smaller than $\frac{1}{4}$. Without loss of generality, we assume that
$$\sup_{t\in [0,s]}\int_M|w(t)|^2=\int_M|w(s)|^2.$$
Hence we obtain that $\sup_{t\in [0,s]}\int_M|w(t)|^2=0,$ i.e., $u\equiv v$ on $[0,s)$.
\end{proof}
\begin{proof}[\bf{Proof of Theorem \ref{thm:main-02}}]
By Theorem \ref{main thm}, we see that there exists a time sequence $\{t_i\}$, $t_i\to +\infty$ as $i\to+\infty$, such that $\frac{\partial u(t_i,\cdot)}{\partial t_i}\to 0$ in $L^2(M,N)$ and $u(\cdot,t_i)$ converges weakly in $W^{2,2}(M,N)$ to a map $u_\infty \in W^{2,2}(M,N)$ with Dirichlet boundary data $u=g$ and $\partial_\nu u=h$ on $\partial M$, where $g\in C^{4+\alpha}(\partial M,N)$ and $h\in C^{3+\alpha}(\partial M,T_gN)$. Denote $u(t_i)=u_i$ and $g_i=-\frac{\partial u_i}{\partial t_i}$, then $\Delta^2u_i+f(u_i)=g_i$. Note that $g_i\to 0$ in $L^2(M,N)$ and hence in $(W^{2,2}(M,N))^\ast$,  by the weak compactness theorem of Zheng \cite{Zh}, we see that $u_\infty \in W^{2,2}(M,N)$ is a biharmonic map. Then we see that $u\in C^{4+\alpha, 1+\frac{\alpha}{4}}(\overline M,N)$ by the interior regularity theorem of Wang \cite{W2} and boundary regularity theorem of Lamm and Wang \cite{Lamm3}.
\end{proof}

\quad\\

\noindent\textbf{Acknowledgement}. Huang is supported by NSF of Shanghai grant 16ZR1423800. Liu is supported by the NSF of China (No.11471299). Luo is supported by the NSF of China (No.11501421), the Postdoctoral Science Foundation of China (No.2015M570660), and the Project-sponsored by SRF for ROCS, SEM. Wang is partially supported by NSF.

 {}

\vspace{0.2cm}\sc

\noindent NYU-ECNU Institute of Mathematical Sciences at NYU Shanghai,
3663 Zhongshan Road North,
Shanghai,
200062,
China; {\tt th79@nyu.edu}
\vspace{0.2cm}\sc

\noindent Department of Mathematics, Tsinghua University, HaiDian Road, BeiJing 100084, China;
Max-planck institut f\"ur mathematik In den naturwissenschaft, Inselstr.22, D-04103, Leipzig, Germany;
{\tt liulei1988@mail.tsinghua.edu.cn}~{\em or}~{\tt leiliu@mis.mpg.de}

\vspace{0.2cm}\sc

\noindent School of mathematics and statistics, Wuhan university, Wuhan 430072, China;
Max-planck institut f\"ur mathematik In den naturwissenschaft, Inselstr.22, D-04103, Leipzig, Germany;
{\tt yongluo@whu.edu.cn}~{\em or}~{\tt yongluo@mis.mpg.de}

\vspace{0.2cm}\sc
\noindent Department of Mathematics, Purdue University, West Lafayette, IN 47907;\\ {\tt wang2482@purdue.edu}


\begin{thebibliography}{2}
\bibitem{BVP}
P. Bousquet, A. Ponce and J. V. Schaftingen, \newblock{Strong density for higher order Sobolev spaces into compact manifolds},
\newblock{ J. Eur. Math. Soc.} {\bf17}(2015), 763-817.

\bibitem{BL}
C. Breiner and T. Lamm, \newblock{Quantitative stratification and higher regularity
for biharmonic maps}, \newblock{Manuscripta Math.} {\bf148}(2015), 379-398.

\bibitem{CWY}
S.-Y.A. Chang,  L. Wang and P.C. Yang, \newblock{Regularity of harmonic maps}, \newblock{Commun. Pure
Appl. Math.} {\bf52}(1999), 1099-1111.

\bibitem{Cooper}
 M. Coorper,
    \newblock{Critical $O(d)-$equivatiant biharmonic maps},
    \newblock{Calc. Var. Partial Differential Equations}, {\bf  54}(3) (2015),  2895-2919.

\bibitem{Gastel}
 A. Gastel,
    \newblock{The extrinsic polyharmonic map heat flow in the critical dimension},
    \newblock{Adv. Geom.}, {\bf  6}(4) (2006),  501-521.

\bibitem{GT}
D. Gilbarg and  N. S. Trudinger, \newblock{Elliptic Partial Differential Equations of Second Order,} \newblock{Springer Verlag,
Heidelberg (2001)}.

\bibitem{Hong-Yin}
M. Hong and H. Yin, \newblock{Partial regularity of a minimizer of the relaxed energy for biharmonic maps}, \newblock{J. Funct. Anal.} {\bf262}(2) (2012), 681-718.


\bibitem{Ku}
F. Ku, \newblock{Interior and boundary regularity of intrinsic biharmonic maps to spheres}, \newblock{Pacific J. Math.} {\bf234}(2008), 43-67.

\bibitem{Lamm}
T. Lamm, \newblock{Biharmonischer W\"armefluss}, \newblock{Diplomarbeit}, Universit\"at Freiburg(2001).
 \bibitem{Lamm1}

 T. Lamm,
    \newblock{Heat flow for extrinsic biharmonic maps with small initial energy},
    \newblock{Ann. Global Anal. Geom.}, {\bf  26}(4) (2004),  369-384.
 \bibitem{Lamm4}

 T. Lamm,
    \newblock{ Biharmonic map heat flow into manifolds of nonpositive curvature},
    \newblock{Calc. Var. Partial Differential Equations}, {\bf  22}(4) (2005),  421-445.

\bibitem{Lamm2}
T. Lamm, T. Rivi\'ere,
\newblock{Conservation laws for fourth order systems in four dimensions}, \newblock{Commun. Partial Differ. Equ.} {\bf33}(2008), 245-262.

\bibitem{Lamm3}
T. Lamm and C. Wang, \newblock{Boundary regularity for polyharmonic maps in the critical dimension}, \newblock{Adv. Calc. Var.} {\bf2}(2009), 1-16.

\bibitem{LY}
  L. Liu and H. Yin,
    \newblock{Neck analysis for biharmonic maps}, \newblock{arxiv:1312.4600},
    \newblock{to appear in Math. Z}.

\bibitem{LY-2}
  L. Liu and H. Yin,
    \newblock{On the finite time blow-up of biharmonic map flow in dimension four}, 
    \newblock{Journal of Elliptic and Parabolic Equations}, {\bf1}(2015), p.363-385.

\bibitem{Riviere1}
T. Rivi\'ere, \newblock{Conservation laws for conformally invariant variational problems}, \newblock{Invent.
Math.} {\bf168}(2008), 1-22.

\bibitem{Riviere2}
T. Rivi\'ere and M. Struwe, \newblock{Partial regularity for harmonic maps and related problems,} \newblock{Commun. Pure Appl. Math.} {\bf61}(2008), 451-463.

\bibitem{Scheven}
C. Scheven, \newblock{Dimension reduction for the singular set of biharmonic maps,} \newblock{Adv. Calc.
Var.} {\bf1}(2008), 53-91.

\bibitem{SU}
R. Schoen and K. Uhlenbeck, \newblock{Boundary regularity and the Dirichlet problem for harmonic maps},
\newblock{J. Differential Geom.} {\bf18}(1983), 253-268.

\bibitem{struwe1}
M. Struwe, \newblock{On the evolution of harmonic mappings of Riemannian surfaces}, \newblock{Commun. Math. Helv.} \textbf{60}  (1985), 558-581.

\bibitem{struwe2}
M. Struwe, \newblock{Partial regularity for biharmonic maps, revisited,} \newblock{Calc. Var. Partial Differ.
Equ.} {\bf33}(2008), 249-262.

\bibitem{Strz}
P. Strzelecki, \newblock{On biharmonic maps and their generalizations,} \newblock{Calc. Var. Partial Differ.
Equ.} {\bf18}(2003), 401-432.

\bibitem{W1}
C. Wang, \newblock{Remarks on biharmonic maps into spheres}, \newblock{Calc. Var. Partial Differ.
Equ.} \textbf{21}(2004), 221-242.

\bibitem{W2}
C. Wang, \newblock{Biharmonic maps from $\R^4$ into a Riemannian manifold}, \newblock{Math. Z.} \textbf{247}(2004), 65-87.

\bibitem{W3}
C. Wang, \newblock{Stationary biharmonic maps from $\R^m$ into a Riemannian manifold}, \newblock{Commun.
Pure Appl. Math.} \textbf{57}(2004), 419-444.

 \bibitem{W4}
  C. Wang,
    \newblock{Heat flow of biharmonic maps in dimension four and its application},
    \newblock{Pure. Appl. Math. Quar.}, {\bf  3}(2) (2007),  595-613.

 \bibitem{Zh}
S. Zheng,
\newblock{Weak compactness of biharmonic maps},
\newblock{ Electron. J. Diff. Equ.} {\bf190}(2012), 7pp.

\end{thebibliography}
\end{document}